\documentclass[12pt,reqno]{amsart}

\usepackage[dvipsnames]{xcolor}
\usepackage{a4,url,amssymb,enumerate,tikz,scalefnt,mathrsfs}

\DeclareFontFamily{OT1}{pzc}{}
\DeclareFontShape{OT1}{pzc}{m}{it}{ <-> s*[1.2] pzcmi7t }{}
\DeclareMathAlphabet{\mathpzc}{OT1}{pzc}{m}{it}

\usetikzlibrary{arrows, arrows.meta, automata, trees}
\usetikzlibrary{shapes,shapes.geometric,fit,calc,positioning,automata}

\usepackage{enumitem}
\usepackage{comment}

\makeatletter
\@namedef{subjclassname@2010}{\textup{2020} Mathematics Subject Classification}
\makeatother

\theoremstyle{plain}
\newtheorem{thrm}{Theorem}[section]
\newtheorem{lemma}[thrm]{Lemma}
\newtheorem{prop}[thrm]{Proposition}
\newtheorem{cor}[thrm]{Corollary}
\newtheorem{que}[thrm]{Question}

\theoremstyle{definition}
\newtheorem{defn}[thrm]{Definition}
\newtheorem{ex}[thrm]{Example}

\DeclareMathOperator{\Av}{Av}
\DeclareMathOperator{\col}{col}

\newcommand{\lecons}{\leq_{\textup{cons}}}
\newcommand{\nlecons}{\nleq_{\textup{cons}}}
\newcommand{\congcons}{\cong_{\textup{cons}}}
\newcommand{\Harpoon}[1]{\hspace{-0.3mm}\! \upharpoonright_{#1}}
\newcommand{\Harpoonc}[1]{\hspace{-0.1mm}\! \upharpoonright_{#1}}
\newcommand{\down}{\hspace{-0.9mm} \downarrow}

\newcommand{\lvertb}{\lvert \hspace{0.5mm}}
\newcommand{\lvertm}{\hspace{0.5mm} \lvert \hspace{0.75mm}}
\newcommand{\lverte}{\hspace{0.5mm} \lvert}
\newcommand{\h}{\hspace{0.75mm}}
\newcommand{\hh}{\hspace{0.4mm}}

\DeclareMathOperator{\Eq}{\mathpzc{Eq}}

\newcommand{\cequiv}[3]{\underline{\colorbox{#1}{#2}}_{\hh #3} \hspace{-0.8mm}}

\newenvironment{thmenumerate}{\begin{enumerate}[label=\textup{(\roman*)},leftmargin=10mm]}{\end{enumerate}}
\newenvironment{nitemize}{\begin{itemize}[label=\textbullet, leftmargin=5mm]}{\end{itemize}}

\title[Well quasi-order and atomicity for equivalence relations]{Decidability of well quasi-order and atomicity for equivalence relations under embedding orderings}

\author{V. Ironmonger}
\author{N. Ru\v{s}kuc}
\address{School of Mathematics and Statistics, University of St Andrews, St Andrews, Scotland, UK}
\email{$\{$vli,nr1$\}$@st-andrews.ac.uk}


\date{\today}

\begin{document}

\maketitle

\begin{abstract}
We consider the posets of equivalence relations on finite sets under the standard embedding ordering and under the consecutive embedding ordering. In the latter case, the relations are also assumed to have an underlying linear order, which governs consecutive embeddings. For each poset we ask the well quasi-order and atomicity decidability questions: Given finitely many equivalence relations $\rho_1,\dots,\rho_k$, is the downward closed set $\Av(\rho_1,\dots,\rho_k)$ consisting of all equivalence relations which do not contain any of $\rho_1,\dots,\rho_k$: (a) well-quasi-ordered, meaning that it contains no infinite antichains? and (b) atomic, meaning that it is not a union of two proper downward closed subsets, or, equivalently, that it satisfies the joint embedding property?
\end{abstract}

\section{Introduction}
\label{section intro}

Embedding orderings of different types of combinatorial structures have over the years proved a fruitful field of research 
with a pleasing mix of combinatorics and order theory.  
R. Fra\"{\i}ss\'{e} was the most notable pioneer of this interface between combinatorics and model theory, and his monograph \cite{fraisse00} remains an excellent introduction.
The most prominent structures that have been investigated in this context include graphs, different types of digraphs, permutations and words. 
By considering all finite structures of a given type up to isomorphism, and fixing an appropriate notion of what it means for one structure to embed into another, one is faced with a countably infinite poset which typically has a very complex structure. The challenge then is to try and gain insights into this infinite poset. One way of doing this is by investigating downward closed subsets, which can be equivalently described as avoidance sets, i.e. sets of all structures which do not contain any of a (finite or infinite) list of structures. By restricting to those sets defined by finitely many forbidden substructures one can phrase questions as algorithmic decidability problems.

Thus, fixing a collection of finite combinatorial structures $\mathcal{C}$ and a property $\mathcal{P}$, one can ask whether the following problem is algorithmically decidable:
given finitely many structures $A_1,\dots,A_k\in\mathcal{C}$ does the downward closed set
$\Av(A_1,\dots,A_k)$, which consists of all members of $\mathcal{C}$ which do not contain any of the $A_i$,
 satisfy property $\mathcal{P}$.
 
Two properties that have been much investigated are well quasi-order and atomicity. The former means absence of infinite antichains (as well as infinite descending chains, but this is automatic for collections of finite structures); the latter is equivalent to the joint embedding property -- for any two members in the poset there is a member that contains them both.

Over the years there has been considerable variance in nomenclature in literature for these two properties.
Well quasi-order has been called  finite basis property in Higman's seminal paper \cite{higman}, where he also states that
partial well-order was used by Erd\"{o}s and Rado.
Atomic sets were called ideals in \cite{fraisse00}, as well as ages and directed sets.
We have selected the term well quasi-order as it seems to have become common in recent literature, and atomicity as it expresses the structural significance of this property as can be seen in \cite[Section 3.3]{vatter} in the case of permutations.

There are many papers dealing with this subject matter, most notably on well quasi-order in graphs.
For example, the wqo problem for the  
ordinary subgraph ordering is decidable as an immediate consequence of \cite{ding92}.  The Graph Minor Theorem \cite{rob-sey-04} asserts that under the minor ordering (which is, strictly speaking not an embedding ordering) the set of all graphs is wqo (and hence the wqo problem is trivially decidable).
By way of contrast, the wqo problem is wide open for some other natural orderings, notably the induced subgraph ordering.
The same is true for some special classes or variations of graphs, such as bipartite graphs, digraphs and tournaments, as well as for other combinatorial structures, notably permutations.
One further exception is provided by the case of words over a finite alphabet under the (scattered) subword ordering, where the entire poset is wqo due to the so-called Higman's Lemma \cite{higman}, which we will briefly review in Section \ref{section prelims}, and hence the wqo problem for downward closed classes in this case is trivially decidable.
We refer the reader to the first half of \cite{cherlin11} for a motivational survey; \cite{hucz15} takes a more comparative-combinatorial viewpoint, and \cite{liu20} is the most up to date survey focussing on graphs.
Turning to the atomicity problem, a good introduction into the concept and its structural significance is given in \cite{vatter}.
A recent major result shows that the property is undecidable for the induced subgraph ordering \cite{braun19}. The problem is still open for permutations, but \cite{braun21} shows that atomicity is undecidable for `3-dimensional' permutations, i.e. sets with three linear orders.

An embedding ordering that has recently come to prominence is the so-called \emph{consecutive} ordering. 
For words, this would be the usual consecutive subword (sometimes called factor) ordering.
In permutations, this ordering arises when the entries are required to embed consecutively in one (out of the two available) dimensions/linear orders. 
We refer the reader to \cite{eli16} for a survey of this ordering for permutations, and \cite{eli18} for some recent insights into the structure of the resulting poset.
In general, to be able to define consecutive ordering,
one requires the presence of a linear order in the language for our combinatorial structures.
 In \cite{alm} it is proved that the wqo problem is decidable for the consecutive embedding ordering on words over a finite alphabet. This was taken further in \cite{mr}, where it was proved that wqo is also decidable for the consecutive embedding ordering on permutations, and that atomicity is decidable for consecutive embedding orderings on both words and permutations. In all these cases the key idea is to re-interpret the problem in terms of subpath ordering on the set of all paths in a finite digraph. 

Motivated by these similarities and the desire to gain more understanding into the general behaviour of consecutive orderings, in this paper we take another type of very elementary combinatorial structure, namely equivalence relations. 
Of course, they do not come naturally equipped with a linear order, so to be able to consider consecutive orderings we  add one to the signature for our structures. 
In order to fill a somewhat surprising gap in literature, we also consider the (non-consecutive) embedding ordering for equivalence relations.

Thus we arrive at the topic and content of the present paper:
we investigate the collection of all equivalence relations on finite sets under two orderings -- the standard (or non-consecutive) embedding ordering and the consecutive embedding ordering.   We will consider decidability of the well quasi-order and atomicity problems under each ordering.  Our first result proves that the collection is well quasi-ordered under the non-consecutive embedding order; therefore, the well quasi-order problem is trivially decidable.  For the remaining cases, we will find equivalent conditions to well quasi-order or atomicity, and decidability will follow by showing these conditions to be testable.  The condition for atomicity under the non-consecutive embedding order can be stated now, whereas the conditions for the consecutive embedding order need more introduction and so are not given at this stage.  Hence our main theorems are: 

\vspace{-2mm}
\begin{nitemize}
\item The collection of equivalence relations under the non-consecutive embedding order is well quasi-ordered (Theorem \ref{thrm wqo standard embedding});
\item A collection defined by finitely many forbidden equivalence relations under the non-consecutive embedding order is atomic if and only if, for each forbidden equivalence relation, all of its classes have the same size (Theorem \ref{thrm atomicity non cons}); in particular, the atomicity problem is decidable for such collections (Theorem \ref{decidability atomicity non cons});
\item The well quasi-order problem is decidable for collections defined by finitely many forbidden equivalence relations
 under the consecutive embedding order (Theorems \ref{thrm wqo cdn cons}, \ref{thrm decidability wqo cons});
\item The atomicity problem is decidable for collections defined by finitely many forbidden equivalence relations under the consecutive embedding order (Theorems  \ref{thrm atomicity cons}, \ref{thrm decidability atomicity cons}).
\end{nitemize}


The paper is organised as follows.  We begin by giving some necessary preliminary results and notation in Section \ref{section prelims}.  Following this we look at wqo and atomicity for equivalence relations under the non-consecutive embedding order.  In Section \ref{section wqo noncons} we show that the poset of finite equivalence relations under the non-consecutive embedding order is wqo (Theorem \ref{thrm wqo standard embedding}).  Then in Section \ref{section atomicity noncons} we answer the atomicity problem for the poset of equivalence relations under the non-consecutive embedding order (Theorem \ref{decidability atomicity non cons}).  

Sections \ref{section digraphs}--\ref{section atomicity cons} tackle the well quasi-order and atomicity problems for the poset of finite equivalence relations under the consecutive embedding order.  We will relate the poset of equivalence relations in an avoidance set to the poset of paths in certain finite digraphs.  We rely on results from \cite{mr} which give criteria for these posets of paths to be well quasi-ordered or atomic; these are introduced in Section \ref{section digraphs}.  Section \ref{section factor graph} introduces the technical tools needed to apply these to equivalence relations, and we utilise these tools in Section \ref{section wqo unbounded} to give criteria for wqo in two particular cases.  To tackle the remaining case, in Section \ref{section coloured equivs} we introduce a new poset of coloured equivalence relations; combining these results enables us to answer the wqo problem in general in the affirmative (Theorem \ref{thrm decidability wqo cons}).  Finally, in Section \ref{section atomicity cons} we answer the atomicity problem for the poset of equivalence relations under the consecutive embedding ordering (Theorem \ref{thrm decidability atomicity cons}).
The paper concludes with some remarks and open problems in Section \ref{sec:conc}.

\section{Preliminaries}
\label{section prelims}

An \emph{equivalence relation}, considered as a relational or combinatorial structure, is 
simply a pair $(X, \rho)$, where $X$ is a set and $\rho\subseteq X\times X$ is a binary relation which is reflexive, symmetric and transitive.  

Often we will denote an equivalence relation $(X,\rho)$ as a list of its equivalence classes, separated by vertical bars. 
For example, $\lvertb1 \lvertm 2 \lvertm ...\lvertm n \lverte$ is the equality relation on $\{1, \dots, n\}$, whereas
$\lvertb1 \h 2\dots n \lverte$ is the full relation.  The equivalence class of an element $x \in X$ is denoted $\rho_{x}$.

We will consider two posets of equivalence relations; the first of them will use the standard embedding ordering on relational structures:
\begin{defn}
\label{defn noncons}
The (\emph{non-consecutive})  \emph{embedding ordering} on equivalence relations is given by $(X, \rho) \leq (Y, \sigma)$ if and only if there is an injective function $f:X \rightarrow Y$ such that $(x, y)\in \rho$ if and only if $(f(x), f(y)) \in \sigma$ for all $x,y\in X$. We also say that $(X,\rho)$ is a \emph{sub-equivalence relation} of $(Y, \sigma)$, and that $f$ is an \emph{embedding} of 
$(X,\rho)$ into $(Y,\sigma)$.
\end{defn}

Associated with this definition of embedding is the following definition of isomorphism.  Two equivalence relations $(X,\rho)$ and $(Y,\sigma)$ are \emph{isomorphic} if there is a bijection $f:X\rightarrow Y$ such that for all $x,y\in X$ we have $(x,y)\in\rho$ if and only if $(f(x),f(y))\in\sigma$, and we will write $(X, \rho) \cong (Y, \sigma)$; this is equivalent to $\rho$ and $\sigma$ having the same number of equivalence classes of any size.  Observe that if $X, Y$ are finite, then $(X, \rho) \cong (Y, \sigma)$ if and only if $(X, \rho) \leq (Y, \sigma)$ and $(Y, \sigma) \leq (X, \rho)$.  We will consider isomorphic equivalence relations to be equal, and gather the finite ones into a set $\Eq$.  It can be seen that every equivalence relation on a finite set is isomorphic to an equivalence relation on a subset of $\mathbb{N}$ so, without loss of generality, from now on we limit our considerations to equivalence relations of this form.  In fact, we will almost always work with equivalence relations on the set $[n]=[1, n]=\{1, \dots, n\}$ for some $n \in \mathbb{N}$.  With these conventions, the set of equivalence relations is an infinite poset under the non-consecutive embedding order, denoted by $(\Eq, \leq)$.   

Our second poset will use a consecutive embedding ordering, for which we will need our underlying sets to be linearly ordered.

\begin{defn}
\label{defn contiguous map}
Let $X=\{x_{1}, \dots, x_{n}\}$ and $Y=\{y_{1}, \dots, y_{m}\}$, and let $\leq_{X}$ and $\leq_{Y}$ be linear orders on $X$ and $Y$ respectively so that $x_{1} \leq_{X} x_{2} \leq_{X} \dots \leq_{X} x_{n}$ and $y_{1} \leq_{Y} y_{2} \leq_{Y} \dots \leq_{Y} y_{m}$. A mapping $f:X \rightarrow Y$ is \emph{contiguous} (or \emph{consecutive}) if there exists $k$ such that
$f(x_{i})=y_{k+i-1}$ for all $i\in [1, n]$. 
\end{defn}

Note that contiguous maps are always injective. 

\begin{defn}
\label{defn cons order}
Let $(X, \rho)$, $(Y, \sigma)$ be equivalence relations, and let $\leq_X$ and $\leq_Y$ be linear orders on $X$ and $Y$.
We say that $(X, \rho)$ \textit{embeds consecutively} in $(Y, \sigma)$ if there is a consecutive embedding $f:X \rightarrow Y$.  This is written $(X, \rho) \lecons (Y, \sigma)$, and we say that $(X, \rho)$ is a \emph{consecutive sub-equivalence relation} of $(Y, \sigma)$.
\end{defn}

As with the non-consecutive embedding ordering, we have a notion of isomorphism under the consecutive embedding ordering.  Two equivalence relations $(X,\rho)$ and $(Y,\sigma)$ are \emph{isomorphic} if there is a contiguous bijection $f:X\rightarrow Y$ such that for all $x,y\in X$ we have $(x,y)\in\rho$ if and only if $(f(x),f(y))\in\sigma$; this is written $(X, \rho) \congcons (Y, \sigma)$.  If $X, Y$ are finite, then $(X, \rho) \congcons (Y, \sigma)$ if and only if $(X, \rho) \lecons (Y, \sigma)$ and $(Y, \sigma) \lecons (X, \rho)$.  Again, we will consider isomorphic relations to be equal and gather the finite ones into a set $\overline{\Eq}$.  And again note that every equivalence relation on a finite set is isomorphic to an equivalence relation on a finite subset of $\mathbb{N}$, where we take the linear order to be the natural order.  Again, without loss of generality, we will restrict our considerations to equivalence relations of this type, and we will almost always work with equivalence relations with underlying set $[1, n]$ for some $n \in \mathbb{N}$.  With these conventions, the set of equivalence relations is an infinite poset under the consecutive embedding order, denoted by $(\overline{\Eq}, \lecons)$.

If $(X, \rho)$ is an equivalence relation on $n$ points, we will define the \textit{length} of $\rho$ to be $|\rho|=n$.

\begin{ex}
We have
$\lvertb1\h 2\lvertm3\lverte \lecons \hspace{-1mm} \lvertb1\lvertm2\h 3\lvertm4\lverte$ via the contiguous embedding
$1\mapsto 2$, $2\mapsto 3$, $3\mapsto 4$;
therefore $\lvertb1\h 2\lvertm3\lverte \leq \hspace{-1mm} \lvertb1\lvertm2\h 3\lvertm4\lverte$ as well.
Similarly, $\lvertb1\h 2 \lvertm 3\lverte \leq \hspace{-1mm} \lvertb1\lvertm2\h 4\lvertm3\lverte$, via the embedding
$1\mapsto 2$, $2\mapsto 4$, $3\mapsto 3$.
This embedding is clearly not contiguous, and it is easy to check that neither of the two possible contiguous mappings $[3]\rightarrow [4]$ are embeddings; therefore 
$\lvertb1\h 2\lvertm3\lverte \nlecons \hspace{-1mm} \lvertb1\lvertm2\h 4\lvertm3\lverte$.
Finally, $\lvertb1\h 2 \lvertm 3\lverte \nleq \hspace{-1mm} \lvertb1\lvertm2\lvertm3\lvertm4\lvertm5\lverte$ since it is not possible to map the class of size two injectively to a class of size one.
\end{ex}

We will be interested not only in the two posets $(\Eq,\leq)$ and $(\overline{\Eq},\lecons)$, but also their \emph{downward closed subsets}, for which we establish the basic terminology now.

\begin{defn}
\label{defn downward closed}
Let $(X, \leq)$ be a poset and $Y \subseteq X$.  We say that $Y$ is \emph{downward closed} if whenever $x \in Y$ and $y \leq x$ we have that $y \in Y$.
\end{defn}

\begin{defn}
\label{defn Av classes}
Let $(X, \leq)$ be a poset and $B \subseteq X$.  The \textit{avoidance set of $B$ under the order $\leq$} is the downward closed set 
\begin{equation*}
\Av(B)=\{x\in X: y \nleq x \hspace{3mm} \forall y \in B\},
\end{equation*}
the set of elements which \emph{avoid} $B$.  
\end{defn}

If $C \subseteq X$ is downward closed, then it can be expressed as an avoidance set $C=\Av(B)$ for some set $B$, e.g. $B=X \backslash C$.  Moreover, if $X$ has no infinite descending chains, as is the case with $(\Eq,\leq)$ and $(\overline{\Eq},\lecons)$,  we can take $B$ to be the set of minimal elements of $X \backslash C$; this choice of $B$ is the unique antichain such that $C=\Av(B)$, and will be called the \emph{basis} of $C$.  In this case, if $B$ is finite, we say that $C$ is \emph{finitely based}.

\begin{ex}
In $(\Eq,\leq)$ the downward closed set $\Av(\lvertb1\h 2 \lverte)$ consists of all identity equivalence relations, whereas in $(\overline{\Eq},\lecons)$ the set $\Av(\lvertb1\h 2 \lverte)$ consists of all equivalence relations in which no two consecutive elements belong to the same equivalence class.
By way of contrast, $\Av(\lvertb1\lvertm 2 \lverte)$ consists of all full relations in each of the posets.
\end{ex}

We will be investigating two structural, order theoretic properties for the posets $(\Eq,\leq)$ and $(\overline{\Eq},\lecons)$ and their downward closed sets, namely well quasi-order and atomicity.
The latter is quicker to introduce, so we do this first.

\begin{defn}
\label{defn atomicity}
A downward closed set $C$ in a poset $(X,\leq)$ is \emph{atomic} if $C$ cannot be expressed as a union of two downward closed, proper subsets.
\end{defn}

The following equivalent formulation is well-known, and easy to prove directly; 
see \cite[Section 2.3.11]{fraisse00} or \cite[Theorem 7.1]{hodges-mt}:

\begin{prop}
\label{prop jep}
A downward closed set $C$ in a poset $(X, \leq)$ is atomic if and only if $C$ satisfies the \emph{joint embedding property} (\emph{JEP}): for every pair of elements $x, y \in C$ there is an element $z \in C$ such that $x\leq z$ and $y \leq z$. 
\end{prop}

As mentioned in the Introduction, atomic sets have appeared in literature under several different names, in addition to the JEP,
with terms such as
ideals, ages and directed sets all used in \cite{fraisse00}.
Examples of atomic sets include the avoidance set of any single relation of arity $\geq 2$ (\cite[Section 8.1.2]{fraisse00}),
the avoidance set of a single word under the subword ordering, or the avoidance set of a single permutation under the subpermutation ordering.

\begin{ex}
Consider the poset $(\Eq,\leq)$.  The entire poset is atomic.  To see this, we check the JEP.  So let $(X,\rho)$ and $(Y,\sigma)$ be arbitrary. Without loss assume that $X$ and $Y$ are disjoint. On the set $X\cup Y$ define a relation $\rho\oplus\sigma$ by 
\[
(x,y)\in \rho\oplus\sigma\quad \Leftrightarrow\quad (x,y\in X\text{ and } (x,y)\in \rho)\text{ or } (x,y\in Y\text{ and } (x,y)\in \sigma),
\]
and it is clear that $(X\cup Y,\sigma\oplus\rho)$ embeds both $(X,\sigma)$ and $(Y,\rho)$.
On the other hand consider the downward closed set $\Av(\lvertb1\h 2 \lvertm 3 \lverte)$.
If $(X,\rho)$ belongs to this set, and if $\rho$ has a non-singleton class, then $\rho$ cannot have another class.
Therefore $\Av(\lvertb1\h 2 \lvertm 3 \lverte)=\Av(\lvertb 1\h 2\lverte)\cup\Av(\lvertb 1\lvertm 2 \lverte)$, and so
$\Av(\lvertb1\h 2 \lvertm 3 \lverte)$ is not atomic.
\end{ex}

\begin{defn}
\label{defn:atomprob}
The \emph{atomicity problem} for finitely based downward closed sets of a poset $(X,\leq)$ is the algorithmic decidability problem, which takes as its input a finite set $B\subseteq X$ and asks whether or not $\Av(B)$ is atomic.
\end{defn}

As indicated in the introduction, in this paper we are going to prove that the atomicity problem is decidable in both
$(\Eq,\leq)$ (Theorem  \ref{decidability atomicity non cons}) and $(\overline{\Eq},\lecons)$ (Theorem \ref{thrm decidability atomicity cons}).

We now move to the property of being well quasi-ordered.

\begin{defn}
\label{defn wqo}
A poset $(X, \leq)$ is \emph{well quasi-ordered (or wqo)} if it contains no infinite antichains and no infinite descending chains. 
\end{defn}

Note that because embedding orderings respect size, and we are dealing with finite structures, the non-existence of infinite descending chains is automatic, and so the wqo property is equivalent to the absence of infinite antichains.
Also note that
even though we are working with partially ordered sets, we will use the term \emph{well quasi-ordered}, rather than \emph{well partially ordered} or \emph{partially well ordered}, in order to keep with the prevailing usage in the literature.

\begin{ex}
\label{ex equiv relns not wqo cons}
In Section \ref{section wqo noncons}
we prove that $(\Eq,\leq)$ is wqo. 
The poset of equivalence relations under the consecutive embedding order is not well quasi-ordered;
an infinite antichain is given by the set $\bigl\{\lvertb 1 \h n \lvertm 2 \lvertm \dots \lvertm n-1 \lverte \hspace{1.3mm}: n=1, 2, \dots \bigr\}$.
\end{ex}

\begin{defn}
\label{defn:wqoprob}
The \emph{wqo problem} for finitely based downward closed sets of a poset $(X,\leq)$ is the algorithmic decidability problem, which takes as its input a finite set $B\subseteq X$ and asks whether or not $\Av(B)$ is wqo.
\end{defn}

We make the following straightforward observations:

\begin{lemma}
\label{lemma prelims wqo}
\begin{thmenumerate}
\item \label{lemma subset of wqo set}
Any subset of a wqo set is also wqo. 
\item \label{lemma finite set and wqo set 2}
If $X$ is a poset and $Y$ is a finite subset of $X$, then $X$ is wqo if and only if $X\backslash Y$ is wqo. 
\end{thmenumerate}
\end{lemma}

Since $(\Eq,\leq)$ is wqo, it follows that all its downward closed subsets are also wqo, and the wqo problem is trivial in this case.
In the case of $(\overline{\Eq},\lecons)$ the problem is non-trivial, and we will prove it is decidable in Section \ref{section wqo in general}.

A key tool in establishing wqo in different contexts is the so-called Higman's Lemma. The result was originally proved in a general context of universal algebra, but we only give a specialisation to free semigroups and an immediate corollary that will be of use to us.

\begin{defn}
\label{defn dom order}
Let $(X, \leq_X)$ be a poset and $X^{*}$ be the set of words over $X$.  We define the \emph{domination order} $\leq_{X^{*}}$ on $X^{*}$ by 
\begin{equation*}
x_{1}x_{2}...x_{k} \leq_{X^{*}} y_{1}y_{2}...y_{l}
\end{equation*}
if and only if there is a subsequence $j_{1} < j_{2}<...<j_{k}$ of $[1, l]$ such that $x_{t}\leq_{X} y_{j_{t}}$ for $t=1, ..., k$. 
\end{defn}

\begin{lemma}[Higman's Lemma, \cite{higman}]
\label{lemma Higman}
Let $(X, \leq_{X})$ be a poset and $(X^{*}, \leq_{X^{*}})$ be the poset of words over $X$ with the domination order.  If $X$ is wqo under $\leq_{X}$ then $X^{*}$ is wqo under $\leq_{X^{*}}$.
\end{lemma}

\begin{cor}
\label{cor to Higman}
The set of finite sequences of natural numbers is wqo under the ordering given by
\begin{equation*}
(x_{1},..., x_{k}) \leq (y_{1}, ..., y_{l})\quad \Leftrightarrow \quad
x_i\leq y_{j_i} \text{ for some } 1\leq j_1<\dots<j_k\leq l.
\end{equation*}
\end{cor}

\section{WQO under the non-consecutive embedding ordering}
\label{section wqo noncons}

The purpose of this section is to show that the poset of equivalence relations under the non-consecutive embedding order is wqo.  To do this we need to show that this poset contains no infinite antichains.  

We begin by introducing another ordering on finite sequences of natural numbers, the prefix domination order.  We then express the non-consecutive embedding order on equivalence relations in terms of both the prefix domination order and the usual domination order on the sizes of its classes.  These results will allow us to tackle the wqo and atomicity problems for our poset by looking at the domination order and prefix domination order respectively.

\begin{defn}
\label{defn prefix dom}
Let $\sigma=(\sigma_{1}, \dots, \sigma_{k})$, $\tau=(\tau_{1}, \dots, \tau_{l})$ be finite sequences of natural numbers.  We say that $\sigma$ is related to $\tau$ under the \emph{prefix domination order}, written $\sigma \leq_{p} \tau$, if and only if $k \leq l$ and $\sigma_{i} \leq \tau_{i}$ for $i=1, \dots, k$.
\end{defn}

In what follows we will use the term `decreasing' for a sequence in which each entry is less than or equal to the entry preceding it.

\begin{lemma}
\label{lemma 2 orders on seqs}
If $\sigma, \tau$ are finite decreasing sequences of natural numbers then $\sigma \leq \tau$ under the domination ordering if and only if $\sigma \leq_{p} \tau$.
\end{lemma}

\begin{proof}
($\Leftarrow$) This direction follows immediately from the definitions. 
($\Rightarrow$) Let $\sigma \leq \tau$ under the domination order and suppose $\sigma=(\sigma_{1}, \dots, \sigma_{k})$, $\tau=(\tau_{1}, \dots, \tau_{l})$.  Then there is a sequence $j_{1}<\dots<j_{k}$ from $[1, l]$ such that $\sigma_{i} \leq \tau_{j_{i}}$ for each $i$.  Since $j_{i} \geq i$, $\tau_{j_{i}} \leq \tau_{i}$ for each $i$.  This means that $\sigma_{i} \leq \tau_{j_{i}} \leq \tau_{i}$ for each $i$, so it is true that $\sigma \leq_{p} \tau$.
\end{proof}

Given an equivalence relation $(X, \rho)$ with equivalence classes $C_{1}, ..., C_{N}$ in decreasing size order, we will assign to $(X, \rho)$ the sequence of natural numbers $(|C_{1}|, |C_{2}|, ... |C_{N}|)$ and denote this sequence $\pi(X, \rho)$, or just $\pi(\rho)$. 

\begin{lemma}
\label{lemma noncons cdn combined}
If $(X, \rho)$, $(Y, \sigma)$ are equivalence relations, the following are equivalent:
\begin{thmenumerate}
\item
\label{cdn cmd 1} 
$(X, \rho) \leq (Y, \sigma)$ under the non-consecutive embedding order;
\item
\label{cdn cmd 2}
$\pi(X, \rho) \leq \pi(Y, \sigma)$ under the domination order;
\item
\label{cdn cmd 3}
$\pi(X, \rho) \leq_{p} \pi(Y, \sigma)$ under the prefix domination order.
\end{thmenumerate}
\end{lemma}

\begin{proof}
Let $(X, \rho)$, $(Y, \sigma)$ have equivalence classes $C_{1}, ..., C_{n}$ and $K_{1}, ...K_{m}$ respectively, listed in order of decreasing size.  

\ref{cdn cmd 3} $\Rightarrow$ \ref{cdn cmd 2} If $\pi(X, \rho) \leq_{p} \pi(Y, \sigma)$ then by Lemma \ref{lemma 2 orders on seqs} we have that $\pi(X, \rho) \leq \pi(Y, \sigma)$ under the domination order. 

\ref{cdn cmd 2} $\Rightarrow$ \ref{cdn cmd 1} Suppose $\pi(X, \rho) \leq \pi(Y, \sigma)$ under the domination order, so there is a sequence $j_{1}<j_{2}<...<j_{n}$ of $[1, m]$ such that $|C_{t}|\leq|K_{j_t}|$ for $t\in[1, n]$. This means that we can define an injective mapping $f$ from the set of equivalence classes of $\rho$ to those of $\sigma$ by sending $C_{t}$ to $K_{j_{t}}$ for each $t$.  
This is an injective function $f:X \rightarrow Y$ such that if $x, y \in X$ then
\begin{equation*}
x \leq_{\rho} y \Leftrightarrow f(x) \leq_{\sigma} f(y).
\end{equation*}
Therefore, $(X, \rho) \leq (Y, \sigma)$ as required.

\ref{cdn cmd 1} $\Rightarrow$ \ref{cdn cmd 3}  Suppose $(X, \sigma) \leq (Y, \rho)$, so there is an injective function $f:X \rightarrow Y$ that preserves equivalence classes.  It is clear that $n \leq m$ and $f$ induces an injective mapping 
$f^{\prime}:[1, n] \rightarrow [1, m]$
such that $f$ maps $C_{i}$ to $K_{f^{\prime}(i)}$ for $i=1, \dots, n$.  
Note that $|C_i|\leq |K_{f^\prime(i)}|$ for $i=1, \dots, n$.
Consider an arbitrary $t \in [1, n]$.
If $t \leq f^{\prime}(t)$ then $|C_{t}| \leq |K_{f^{\prime}(t)}| \leq |K_{t}|$ as required.  If $t > f^{\prime}(t)$ then there exists $j < t$ such that $f^{\prime}(j) \geq t$.  Then we have that $|C_{t}| \leq |C_{j}| \leq |K_{f^{\prime}(j)}| \leq |K_{t}|$, so $\pi(X, \rho) \leq_{p} \pi(Y, \sigma)$, completing the proof.
\end{proof}

\begin{thrm}
\label{thrm wqo standard embedding}
The poset of equivalence relations on finite sets under the non-consecutive embedding order is well quasi-ordered.
\end{thrm}

\begin{proof}
Aiming for a contradiction, suppose that there is an infinite antichain of equivalence relations $(X_{1}, \rho_{1}), (X_{2}, \rho_{2}), \dots $.  Applying $\pi$ to each element of this antichain gives the sequence $\pi(X_{1}, \rho_{1}), \pi(X_{2}, \rho_{2}), \dots $of finite sequences of natural numbers.  By Corollary \ref{cor to Higman}, finite sequences of natural numbers are well quasi-ordered, so $\pi(X_{i}, \rho_{i}) \leq \pi(X_{j}, \rho_{j})$ for some $i$ and $j$.  Then by Lemma \ref{lemma noncons cdn combined}, $(X_{i}, \rho_{i}) \leq (X_{j}, \rho_{j})$, a contradiction.  We conclude that the poset of equivalence relations under the non-consecutive embedding order is wqo.
\end{proof}

For completeness we record the following immediate result, obtained by combining Theorem \ref{thrm wqo standard embedding}
and Lemma \ref{lemma prelims wqo}\ref{lemma subset of wqo set}:

\begin{cor}
\label{cor wqo noncons Av classes}
All finitely based avoidance sets of equivalence relations under the non-consecutive embedding order are wqo. \qed
\end{cor}

\section{Atomicity under the non-consecutive embedding ordering}
\label{section atomicity noncons}
Now we will consider atomicity for equivalence relations under the non-consecutive embedding ordering, and all avoidance sets in this section will be under this order.  We begin with a couple of illustrative examples and then come to the main results of this section.

\begin{ex}
\label{ex uniform classes atomic} 
Consider the avoidance set $C=\Av( \lvertb 1\h 2\h 3 \lvertm 4\h 5\h 6 \lverte)$; it consists of  all equivalence relations with at most one class of size $\geq3$.  Take two elements $\sigma, \rho \in C$ with equivalence classes $C_{1}, \dots, C_{m}$ and $K_{1}, \dots, K_{n}$ respectively, listed in order of decreasing size.  Let $p=\max\{|C_{1}|, |K_{1}|\}$ and $t=\max\{m, n\}$.  Now let $\theta$ be an equivalence relation with $t-1$ classes of size 2 and one class of size $p$, so $\pi(\theta)=(p, 2, 2, \dots, 2)$.  Clearly $\theta\in C$, and by Lemma \ref{lemma noncons cdn combined} we also have that $\sigma, \rho \leq \theta$.  This means that $C$ satisfies the JEP and therefore is atomic by Proposition \ref{prop jep}.
\end{ex}

\begin{ex}
\label{ex non-uniform classes not atomic} 
Consider the avoidance set $C=\Av( \lvertb 1\h 2 \lvertm 3\lverte)$, which contains the equivalence relations $\sigma=\lvertb 1 \lvertm 2 \lverte$ and $\rho=\lvertb 1\h 2 \lverte$. Any equivalence relation containing both $\sigma$ and $\rho$ must contain $\lvertb 1\h 2 \lvertm 3\lverte$ and so cannot be in $C$.  Hence $C$ does not satisfy the JEP and therefore by Proposition \ref{prop jep}, $C$ is not atomic.
\end{ex}

\begin{defn}
\label{defn uniform relation}
An equivalence relation is \textit{uniform} if all its equivalence classes are the same size.
\end{defn}

\begin{thrm}
\label{thrm atomicity non cons}
If $C=\Av(B)$ is a finitely based avoidance set with basis $B$, then $C$ is atomic under the non-consecutive embedding order if and only if all elements of $B$ are uniform.
\end{thrm}

\begin{proof}
 ($\Rightarrow$) We prove the contrapositive, that if there is a non-uniform element $\sigma \in B$ then $C$ is not atomic.  Suppose such a non-uniform element $\sigma$ exists with equivalence classes $C_{1}, \dots, C_{m}$, listed in order of decreasing size.  Let $|C_{1}|=k>1$ and suppose that $\sigma$ has $n$ classes of size $k$, so $|C_{1}|= \dots = |C_{n}|=k$. Let $\alpha$ be a uniform equivalence relation consisting of $n$ classes $D_{1}, \dots, D_{n}$ of size $k$.  Let $\beta$ be an equivalence relation identical to $\sigma$ but with all classes of size $k$ replaced by classes of size $k-1$, and let the equivalence classes of $\beta$ be $E_{1}, \dots, E_{m}$ in decreasing size order.  Since $\alpha, \beta \lneq \sigma$ and $B$ is a basis, we have that $\alpha, \beta \in C$.
 
Suppose there is an equivalence relation $\theta \in C$ such that $\alpha, \beta \leq \theta$.  Say the equivalence classes of $\theta$ are $F_{1}, \dots, F_{l}$ in decreasing size order.   Since $\alpha \leq \theta$, by Lemma  \ref{lemma noncons cdn combined}, $|F_{i}| \geq |D_{i}| = |C_{i}|$ for $i=1, \dots, n$.  Similarly, since $\beta \leq \theta$, Lemma \ref{lemma noncons cdn combined} gives that $|F_{i}| \geq |E_{i}| = |C_{i}|$ for $i=n+1, \dots, m$.  Therefore $|F_{i}| \geq |C_{i}|$ for all $i$ and so $\sigma \leq \theta$ by Lemma \ref{lemma noncons cdn combined}, which is a contradiction, so $C$ does not satisfy the JEP and so is not atomic.

($\Leftarrow$) Now suppose that all the elements of $B=\{\rho_{1}, ..., \rho_{n}\}$ are uniform.  Take two relations $\alpha, \beta \in C$.  Suppose $\pi(\alpha)=(a_{1}, \dots, a_{k})$ and $\pi(\beta)=(b_{1}, \dots, b_{l})$.  Without loss of generality, assume $k\geq l$.  Let $\gamma$ be an equivalence relation with $\pi(\gamma)=(c_{1}, \dots, c_{k})$, where 
\begin{equation*}
  c_{i} =
    \begin{cases}
      \max\{a_{i}, b_{i}\}, & i \leq l\\
      a_{i}, & i>l\\
    \end{cases}       
\end{equation*}

Lemma \ref{lemma noncons cdn combined} immediately gives that $\alpha, \beta \leq \gamma$.  To give atomicity, we will show that $\gamma \in C$.  Aiming for a contradiction, suppose that $\rho_{j} \leq \gamma$ for some $j$.  Suppose $\pi(\rho_{j})=(p, p, \dots, p)$, with length $q$; by Lemma \ref{lemma noncons cdn combined}, we have that $p \leq c_{i}$ for $i=1, \dots q$.  Without loss of generality, suppose $c_{q}=a_{q}$.  Then, since $\pi(\alpha)$ is decreasing, for every $i=1, \dots, q$ we have that $p\leq c_{q}=a_{q}\leq a_{i}$.  This implies $\rho_{j} \leq \alpha$, a contradiction.  Hence, $\gamma \in C$ is an equivalence relation containing both $\alpha$ and $\beta$, so $C$ satisfies the JEP and therefore is atomic by Proposition \ref{prop jep}.
\end{proof}

\begin{thrm}
\label{decidability atomicity non cons}
It is decidable whether an avoidance set $C=\Av(B)$ of equivalence relations under the non-consecutive embedding order is atomic.
\end{thrm}

\begin{proof}
Firstly, if $B$ is not a basis, we can reduce it to a basis by removing non-minimal elements.  Then it easy to check whether all the basis elements are uniform, meaning that the condition of Theorem \ref{thrm atomicity non cons} is decidable and hence atomicity is decidable for avoidance sets of equivalence relations under the non-consecutive embedding order.
\end{proof}

\section{Digraphs: definitions and some useful results}
\label{section digraphs}

In this section we state some necessary definitions related to digraphs and give two results from \cite{mr} which will be used in later sections. In the terminology of this paper, a digraph is a structure with a single binary relation. However, to help the intuition and visualisation we will use terminology more familiar from graph theory, which we introduce now.

\begin{defn}
\label{defn digraph}
A \emph{digraph} $G$ is a pair $(V, E)$, where $V$ is a set of \emph{vertices} and $E$ is a set of \emph{edges}, which are ordered pairs of vertices.  If $(u, v)\in E$, we say that $u$ and $v$ are \textit{neighbours} and that $u$ and $v$ are \emph{incident} to the edge $(u, v)$.
\end{defn}

\begin{defn}
\label{defn paths}
A \emph{path} in a digraph $(V, E)$ is an ordered sequence of vertices, written $v_{1} \rightarrow v_{2} \rightarrow \dots \rightarrow v_{n}$, where $(v_{i}, v_{i+1}) \in E$ for $i=1, \dots n-1$.  The \emph{length} of such a path is $n-1$.  The \emph{start vertex} and \emph{end vertex} are $v_{1}$ and $v_{n}$ respectively.  
A \emph{simple path} is a path whose vertices are all distinct.
 A \emph{cycle} is a path of length at least one with $v_{1}=v_{n}$. 
 A \emph{simple cycle} is a cycle $v_{1} \rightarrow v_{2} \rightarrow \dots \rightarrow v_{n}$ in which   
 $v_1,\dots,v_{n-1}$ are distinct.
\end{defn}

\begin{defn}
\label{defn concatenation}
Let  $\pi=v_{1} \rightarrow v_{2} \rightarrow \dots \rightarrow v_{n}$ and $\eta=u_{1} \rightarrow u_{2} \rightarrow \dots \rightarrow u_{k}$ be paths in a digraph and suppose $v_{n}=u_{1}$.  Then $\pi$ and $\eta$ can be \emph{concatenated} to produce a new path 
\begin{equation*}
\pi\eta=v_{1} \rightarrow v_{2} \rightarrow \dots \rightarrow v_{n} = u_{1} \rightarrow u_{2} \rightarrow \dots \rightarrow u_{k}.
\end{equation*}
If $\xi$ is a cycle, we will write the concatenation of $\xi$ with itself $m$ times 
as $\xi^{m}$.
\end{defn}

\begin{defn}
\label{defn in/out degree}
The \emph{in-degree} of a vertex $v$ in a digraph $G$ is the number of vertices $u$ such that $(u, v)$ is an edge in $G$.  Similarly, the \emph{out-degree} of $v$ is the number of vertices $u$ such that $(v, u)$ is an edge in $G$.
\end{defn}

\begin{defn}
\label{defn in-out cycle}
A cycle in a digraph is an \emph{in-cycle} if at least one vertex has in-degree two or more but all vertices have out-degree one.  A cycle is an \emph{out-cycle} if all vertices have in-degree one but at least one vertex has out-degree two or more.  A cycle is an \textit{in-out cycle} if it contains at least one vertex of in-degree two or more and at least one vertex of out-degree two or more.  
\end{defn}

\begin{defn}
\label{defn connected}
A digraph $G$ is \emph{strongly connected} if there is a path between any pair of vertices in $G$.
\end{defn}

\begin{defn}
\label{defn bicycle}
A digraph is a \emph{bicycle} if it consists of two disjoint simple cycles connected by a simple path, where only the start and end vertices of the path are in either cycle.  We refer to the first cycle as the \emph{initial cycle} and to the last cycle as the \emph{terminal cycle}.  Either of the cycles can be empty, and if one cycle is empty then the connecting path may be absent as as well.  However, if neither cycle is empty then the connecting path must have length at least one.  
\end{defn}

\begin{defn}
\label{defn subpath}
A \emph{subpath} of a path $v_{1} \rightarrow v_{2} \rightarrow \dots \rightarrow v_{n}$ is any path $v_{i} \rightarrow v_{i+1} \rightarrow \dots \rightarrow v_{k}$ with $1 \leq i\leq k \leq n$.
\end{defn}

\begin{defn}
\label{defn subgraph order}
Let $G$ be a digraph.  We define the \emph{subpath order} on the set of paths in $G$ as follows.  If $\pi, \eta$ are paths in $G$ then $\pi \leq \eta$ if and only if $\pi$ is a subpath of $\eta$.
\end{defn}

The set of paths in a digraph forms a poset under the subpath order.  The next two propositions from \cite{mr} give criteria for the poset of paths in a digraph to be well quasi-ordered and atomic.

\begin{prop}[{\cite[Theorem 3.1]{mr}}]
\label{prop digraphs wqo}
The poset of paths in a finite digraph $G$ under the subpath order is wqo if and only if $G$ contains no in-out cycles.\qed
\end{prop}

\begin{prop}[{\cite[Theorem 2.1]{mr}}]
\label{prop digraphs atomic}
The poset of paths in a digraph $G$ under the subpath order is atomic if and only if $G$ is strongly connected or a bicycle.\qed
\end{prop}

\section{The Factor Graph of an Avoidance Set}
\label{section factor graph}
We have already seen that the poset of equivalence relations under the non-consecutive embedding order is wqo.  Now we look at the consecutive embedding order, and from now on all avoidance sets will be under this order so it will be written as $\leq$, rather than $\lecons$, and isomorphisms will be written $\cong$, rather than $\congcons$.

In Example \ref{ex equiv relns not wqo cons} we saw that the poset of equivalence relations under the consecutive embedding order is not wqo, so now we are working towards showing decidability of wqo for avoidance sets. This section introduces the equivalence relation factor graph of an avoidance set $C$ and explores the relationship between the poset of paths in this graph and $C$.  The ideas from this section will then be applied in Section \ref{section wqo unbounded} towards showing decidability of wqo for avoidance sets in general, and in Section \ref{section atomicity cons} to establish decidability of atomicity.

Since we will be working under the consecutive embedding order, recall that all equivalence relations are equipped with a linear order -- the natural order on $\mathbb{N}$.  Given $(X, \rho)$, if $S \subseteq X$, we will denote the restriction of $\rho$ to points in $S$ by $\rho \Harpoon{S}$.  It can be seen that the restriction of $(X, \rho)$ to $S$ yields a consecutive sub-equivalence relation, and any consecutive sub-equivalence relation of $(X, \rho)$ can be expressed as a restriction of $\rho$ to a subset of $X$.  We will write $\rho\down$ to denote the equivalence relation obtained from $\rho$ by changing the smallest element into a 1, the second smallest into a 2, and so on.  In other words, $\rho\down$ is the unique equivalence relation isomorphic to $\rho$ whose underlying set is $[1, |\rho|]$.  For example, if $\rho=\lvertb0\lvertm2\h 3\h 6\hspace{1mm} 11\lvertm4\h5\hspace{1mm}50\lverte$ and $S=\{3, 4, 5, 6\}$ then $\rho \Harpoon{S}=\lvertb3\h6\lvertm 4\h 5\lverte$ and $\rho\down=\lvertb1\lvertm2\h 3\h 6\h 7\lvertm4\h 5\h8\lverte$.  

In what follows, $B$ will be a finite set of equivalence relations, $C=\Av(B)$, and $b=\max\{|\rho|:\rho \in B\}$.  Note that we are not assuming that $B$ is necessarily the basis for $C$ (i.e. that it is an antichain).  However, if $B$ is not a basis it can easily be reduced to one by removing the non-minimal elements.  If $S\subseteq \mathbb{N}$ then $C_{S}$ will denote the set $\{\sigma \in C: |\sigma| \in S\}$.  
 
We begin with the following easy observation, which relates wqo in an avoidance set $C$ to wqo in the subset $C_{[b, \infty)}$.
 
\begin{lemma}
\label{lemma Cb and C}
A finitely based avoidance set $C$ is wqo if and only if $C_{[b, \infty)}$ is wqo.
\end{lemma}
 
\begin{proof}
This follows immediately from Lemma \ref{lemma prelims wqo} \ref{lemma finite set and wqo set 2}, taking $X=C$ and $Y=C_{[1, b-1]}$.
\end{proof}
 
In \cite{mr} de Bruijn graphs are used to show decidability of wqo and atomicity for avoidance sets of words under the contiguous subword ordering.  Furthermore, certain modifications of de Bruijn graphs are used to show decidability of wqo and atomicity for avoidance sets of permutations under the contiguous subpermutation ordering.  Similarly, now we will introduce the \emph{equivalence relation factor graph}, another modification of de Bruijn graphs, which we use to tackle the wqo and atomicity problems for our poset of equivalence relations.  

We define $\overline{\Eq}_{b}$ to be the set of equivalence relations on the set $[1, b]$.  
We define $G_{b}$ to be the digraph with vertex set $\overline{\Eq}_{b}$ and an edge from vertex $\mu$ to vertex $\nu$ if and only if $\mu\Harpoon{[2, b]}\cong \nu\Harpoon{[1, b-1]}$.  We will define the \textit{equivalence relation factor graph} $\Gamma_{B}$ of an avoidance set $C=\Av(B)$ to be the induced subgraph of $G_{b}$ with vertex set $C_{b}=C \cap \overline{\Eq}_{b}$.  From now on we will use the shortening \emph{factor graph} to refer to the equivalence relation factor graph.  Now we give two examples of factor graphs.

\begin{ex}
\label{ex factor graph 1}
Let $B=\{\lvertb1\h2\lvertm3\lverte\}$ and consider $C=\Av(B)$.  The factor graph of this avoidance set is shown in Figure \ref{Av(1,2 b 3) figure}.  Since the maximum length of an element of $B$ is $b=3$, the vertices are all equivalence relations on $3$ points in $C$:  
\begin{equation*}
\lvertb1\lvertm2\lvertm3\lverte, \hspace{3mm} \lvertb1 \h 3\lvertm2\lverte \hspace{3mm}, \lvertb1\lvertm2 \h 3\lverte, \hspace{3mm} \lvertb1 \h 2 \h 3\lverte.
\end{equation*}
There is an edge from $\sigma = \lvertb1\lvertm2\lvertm3\lverte$ to $\rho = \lvertb1 \h 3\lvertm2\lverte$ because 
\begin{equation*}
\sigma \Harpoon{\{2, 3\}} = \lvertb 2 \lvertm 3 \lverte \hspace{2mm} \cong \lvertb 1 \lvertm 2 \lverte = \rho \Harpoon{\{1, 2\}}.
\end{equation*}
\end{ex}

\begin{figure}
\begin{tikzpicture}[> =  {Stealth [scale=1.3]}, thick]
\tikzstyle{everystate} = [thick]
\node [state, shape = ellipse, minimum size = 20pt] (1b2b3) at (0,0) {\small{$\lvertb1\lvertm2\lvertm3\lverte$}};
\node [state, shape = ellipse, minimum size = 20pt] (13b2) at (3, 0) {\small{$\lvertb1\h 3\lvertm2\lverte$}};
\node [state, shape = ellipse, minimum size = 20pt]  (1b23) at (0,-3) {\small{$\lvertb1\lvertm2 \h 3\lverte$}};
\node [state, shape = ellipse, minimum size = 20pt] (123) at (3, -3) {\small{$\lvertb1\h  2 \h 3\lverte$}};

\path[->]
(1b2b3) edge [bend left] node {} (13b2)
(13b2) edge [bend left] node {} (1b2b3)
(1b2b3) edge node {} (1b23)
(13b2) edge node {} (1b23)
(1b23) edge node {} (123)
(1b2b3) edge [loop left] node {} (1b2b3)
(13b2) edge [loop right] node {} (13b2)
(123) edge [loop right] node {} (123)
;
\end{tikzpicture}
\caption{The factor graph of $\Av(\lvertb1 \h 2\lvertm 3 \lverte)$.}
\label{Av(1,2 b 3) figure}
\end{figure}

\begin{ex}
\label{ex factor graph 2}
Let $B=\overline{\Eq}_{4} \backslash X$, where
\begin{equation*}
X=\{ \lvertb1 \h 2 \h 3 \h 4\lverte ,\ \lvertb1 \h 2 \h 3\lvertm4\lverte ,\ \lvertb1\lvertm2 \h 4\lvertm 3 \lverte,\ \lvertb 1 \h 3 \h 4\lvertm2 \lverte,\  \lvertb1 \lvertm2 \h 3 \lvertm 4\lverte ,\ \lvertb1 \h 2 \h 4\lvertm3\lverte\}
\end{equation*}
and consider the avoidance class $C=\Av(B)$.  The factor graph $\Gamma_{B}$ of $C$ is shown in Figure \ref{ex 4.7 figure}.
\end{ex}

\begin{figure}
\begin{tikzpicture}[> =  {Stealth [scale=1.3]}, thick]
\tikzstyle{everystate} = [thick]
\node [state, shape = ellipse, minimum size = 20pt] (1234) at (0,0) {\small{$\lvertb1 \h 2 \h 3 \h 4\lverte$}};
\node [state, shape = ellipse, minimum size = 20pt] (123b4) at (4, 0) {\small{$\lvertb1 \h 2 \h 3\lvertm4\lverte$}};
\node [state, shape = ellipse, minimum size = 20pt]  (1b24b3) at (0,-2) {\small{$\lvertb1\lvertm2 \h 4\lvertm3\lverte$}};
\node [state, shape = ellipse, minimum size = 20pt] (134b2) at (4, -2) {\small{$\lvertb1 \h 3 \h 4\lvertm2\lverte$}};
\node [state, shape = ellipse, minimum size = 20pt] (1b23b4) at (8, -2) {\small{$\lvertb1\lvertm2 \h 3\lvertm4\lverte$}};
\node [state, shape = ellipse, minimum size = 20pt] (124b3) at (8, 0) {\small{$\lvertb1 \h 2 \h 4\lvertm3\lverte$}};

\path[->]
(1234) edge node {} (123b4)
(123b4) edge node {} (124b3)
(1b24b3) edge node {} (134b2)
(134b2) edge node {} (1b23b4)
(1b23b4) edge node {} (124b3)
(1234) edge [loop left] node {} (1234)
(124b3) edge node {} (134b2)
;
\end{tikzpicture}
\caption{The factor graph of $\Av(B)$ from Example \ref{ex factor graph 2}.}
\label{ex 4.7 figure}
\end{figure}

Now we describe the relationship between elements in $C_{[b, \infty)}$ and paths in $\Gamma_{B}$.  Let $\rho\in C_{[b, \infty)}$ and without loss of generality assume $\rho$ has underlying set $[1, n]$.  We associate $\rho$ with the path $\Pi(\rho)$ in $\Gamma_{B}$ given by
\begin{equation}
\label{eq:Pirho}
\rho \Harpoon{[1, b]} \rightarrow \rho \Harpoon{[2, b+1]}\downarrow \rightarrow \dots \rightarrow \rho \Harpoon{[n-b+1, n]}\downarrow.
\end{equation}

On the other hand, if $\pi=\mu_{1} \rightarrow \mu_{2}\rightarrow ...\rightarrow \mu_{k}$ is a path in $\Gamma_{B}$, we associate it with the set of equivalence relations  
\begin{equation*}
\Sigma(\pi)=\{\rho\in C_{[b, \infty)} : \Pi(\rho)=\pi\}.
\end{equation*}
Note that while every element of $C_{[b, \infty)}$ is associated with a unique path in $\Gamma_{B}$, a path in $\Gamma_{B}$ may be associated with several equivalence relations in $C_{[b, \infty)}$.  The following properties follow directly from the definitions.

\begin{prop}
\label{prop properties of paths and relations}
\begin{thmenumerate}
\item 
\label{reln in in relns of path of itself}
If $\rho \in C_{[b,\infty)}$, then $\rho \in \Sigma(\Pi(\rho))$.
\item 
\label{path is path of reln of path}
If $\pi$ is a path in $\Gamma_{B}$ and $\sigma \in \Sigma(\pi)$, then $\Pi(\sigma)=\pi$.\qed
\end{thmenumerate}
\end{prop}

It will be possible to show a close relationship between the subpath order on $\Gamma_{B}$ and the consecutive embedding order on $C=\Av(B)$ which will be key in showing decidability of well quasi-order and atomicity for $C$.  Sometimes we will refer to wqo of the poset of paths in a graph $G$ under the subgraph order simply as wqo of $G$.

\begin{prop}
\label{prop subrelns give rise to subpaths}
If $\sigma, \rho \in C_{[b, \infty)}$ and $\sigma \leq \rho$, then $\Pi(\sigma) \leq \Pi(\rho)$ under the subpath order in $\Gamma_{B}$.
\end{prop}

\begin{proof}
From the definition \eqref{eq:Pirho} of $\Pi(\rho)$, we have that for a contiguous subset $S\subseteq [1,|\rho|]$ 
the path $\Pi(\rho\Harpoon{S})$ is a subpath of $\Pi(\rho)$, and since $\sigma$ is on a contiguous subset of $[1,|\rho|]$, the result follows. 
\end{proof}

We will identify when every path in $\Gamma_{B}$ has a unique associated equivalence relation. It turns out that in this case the converse of Proposition \ref{prop subrelns give rise to subpaths} is also true,  
and the wqo problem for $C$ is  reduced to that of wqo for $\Gamma_{B}$ under the subpath order, which we know is decidable by Proposition \ref{prop digraphs wqo}.
 We will show that this will be true if and only if $\Gamma_{B}$ does not contain some particular vertices, called \emph{ambiguous vertices}.  Then it will remain to tackle the question of wqo separately for factor graphs containing ambiguous vertices.  We will use similar methods to tackle the atomicity problem in Section \ref{section atomicity cons}.  

Given a path $\pi$ in the factor graph $\Gamma_{B}$, we can think of constructing an associated equivalence relation $\sigma \in \Sigma(\pi)$ by reading the vertices in order and adding to $\sigma$ the entries 1 to $n$ so that each vertex of $\pi$ is a consecutive sub-equivalence relation of $\sigma$.  In this way, vertices can be thought of as giving instructions to place the next entry into a particular equivalence class.  If $\pi$ is a path in a factor graph and $|\Sigma(\pi)|>1$, there must be at least one vertex in $\pi$ that gives more than one option for the position of the next entry of an equivalence relation in $\Sigma(\pi)$.  Any such vertex must have its largest entry $b$ in a class of size one, otherwise the class of this entry would be uniquely determined by the classes of previous entries. This informs the next definition.

\begin{defn}
\label{defn special vertex}
A vertex in a factor graph is a \textit{special vertex} if the largest entry is in a class of size one.  
\end{defn}

\begin{ex}
\label{ex special vertices}
We return to the factor graph of $C=\Av(\lvertb1\h2\lvertm3\lverte)$, shown in  Figure \ref{Av(1,2 b 3) figure}; the vertex labeled $\lvertb1\lvertm2\lvertm3\lverte$ is special since 3 is in a class of size one, but the vertex $\lvertb1\lvertm2\h 3\lverte$ is not special since 3 is in a class of size two. 
\end{ex}

We will now give an example showing that some, but not all, special vertices give rise to a choice for the next entry of some associated equivalence relations.

\begin{ex}
\label{ex choices at some vertices}
Consider again the avoidance set $C=\Av(B)$ from Example \ref{ex factor graph 2}, where $B=\overline{\Eq}_{4} \backslash X$ and
\begin{equation*}
X=\{ \lvertb1 \h 2 \h 3 \h 4\lverte ,\ \lvertb1 \h 2 \h 3\lvertm4\lverte ,\ \lvertb1\lvertm2 \h 4\lvertm 3 \lverte,\ \lvertb 1 \h 3 \h 4\lvertm2 \lverte,\ \lvertb1 \lvertm2 \h 3 \lvertm 4\lverte ,\ \lvertb1 \h 2 \h 4\lvertm3\lverte\}.
\end{equation*}
The factor graph $\Gamma_{B}$ of $C$ is shown in Figure \ref{ex 4.7 figure}.  It can be seen that $\lvertb1 \h 2 \h 3 \lvertm 4\lverte$ and $\lvertb1 \lvertm2 \h 3\lvertm4\lverte$ are the only special vertices.  The equivalence relations associated with paths ending at $\lvertb1 \h 2 \h 3\lvertm4\lverte$ only have one class, so it is not possible to add the new entry to an existing class at this vertex.  This forces a new class to be added at $\lvertb1 \h 2 \h 3\lvertm 4\lverte$, meaning that $\lvertb1 \h 2 \h 3\lvertm4\lverte$ is a special vertex which offers no choice in the position of the new entry.  On the other hand, the vertex $\lvertb1\lvertm2 \h 3\lvertm4\lverte$ can give a choice in the position of the next entry of an associated equivalence relation, for example the associated equivalence relations of the path 
 \begin{equation*}
\lvertb1\lvertm2 \h 4\lvertm3\lverte \rightarrow \lvertb1 \h 3 \h 4\lvertm2\lverte \rightarrow \lvertb1\lvertm2 \h 3\lvertm4\lverte
\end{equation*}
 include both $\lvertb1 \h6\lvertm2 \h 4 \h 5\lvertm 3\lverte$ and $\lvertb1\lvertm2 \h 4 \h 5\lvertm3\lvertm6\lverte$.  
\end{ex}

We will now address this distinction in special vertices, introducing ambiguous vertices as those which give a choice in the position of the next entry of at least one associated equivalence relation.  

\begin{defn}
\label{defn inactive class}
Suppose $C=\Av(B)$ is an avoidance set, $b$ is the maximum length of an element in $B$ and $\sigma \in C_{[b, \infty)}$.  A class of $\sigma$ which does not contain any of the largest $b-1$ elements in $\sigma$ will be referred to as an \emph{inactive class}.
\end{defn}

\begin{ex}
\label{ex inactive class}
Consider the equivalence relation 
\begin{equation*}
\sigma=\lvertb1 \h 5\lvertm 2 \h 3\lvertm 4 \h 6 \h 7\lverte \in \Av(\lvertb1\lvertm2\lvertm3\lvertm4\lverte).
\end{equation*}
Here $b=4$ and the largest three elements of $\sigma$ are $7, 6$ and $5$.  The class $\{2, 3\}$ does not contain any of these elements so is an inactive class.  The class $\{4, 6, 7\}$ contains both 6 and 7, so this is not an inactive class.
\end{ex}

\begin{defn}
\label{defn ambiguous vertex}
A special vertex $\nu$ is \textit{ambiguous} if there is a vertex $\mu$ such that $(\mu, \nu)$ is an edge and there exists an equivalence relation $\sigma$ such that $\sigma$ has an inactive class and $\Pi(\sigma)$ ends at $\mu$.
\end{defn}

\begin{ex}
\label{ex ambiguous vertices}
\label{ex ambiguous vertices item 1}
Following the discussion in Example \ref{ex choices at some vertices}, the vertex $\lvertb1 \h 2 \h 3\lvertm4\lverte$ in Figure \ref{ex 4.7 figure} is not ambiguous since any relation associated with a path ending at $\lvertb 1 \h 2\h 3 \h 4 \lverte$ can only have one class, so cannot have any inactive classes.  On the other hand, the vertex $\lvertb1\lvertm2 \h 3\lvertm4\lverte$ is ambiguous since the relation $\lvertb 1 \lvertm 2 \h 4 \h 5 \lvertm 3 \lverte$ with associated path $\lvertb 1 \lvertm 2 \h 4 \lvertm 3 \lverte \rightarrow \lvertb 1 \h 3 \h 4 \lvertm 2 \lverte$ has an inactive class.
\end{ex}

\begin{prop} 
\label{prop special vertex in cycle}
Special vertices in cycles are always ambiguous.
\end{prop}

\begin{proof}
Let $\nu$ be a special vertex in a cycle $\pi$ and let $\mu$ be the vertex preceding $\nu$ in $\pi$.  Without loss of generality, assume that $\pi$ starts and ends at $\nu$.  Suppose $\nu$ has $t$ classes.  Let $\eta$ be the concatenation of $\pi^{t}$ and the subpath of $\pi$ from $\nu$ to $\mu$.  Consider any $\rho \in \Sigma(\eta)$ for which each of the $t+1$ visits to $\nu$ is an instruction to add a new class.  Then $\rho$ has at least $t+1$ classes.  
Since $\nu$ has precisely $t$ classes, one of which is the singleton $\{b\}$, and since $(\mu, \nu)$ is an edge, $\mu$ has at most $t$ classes.  Therefore $\rho$ has more classes than $\mu$, and hence at least one of them must be inactive, proving that $\nu$ is ambiguous.
\end{proof}

Ambiguous vertices are the only vertices which do not necessarily uniquely determine the class of the next entry of an equivalence relation whose associated path contains that vertex.  They allow the next entry to be added to an inactive class, if one exists, or to be the first element in a new class.  This means that more than one equivalence relation may be associated with a path containing an ambiguous vertex.  Therefore, if $\pi$ is a path containing ambiguous vertices, we can describe an equivalence relation $\sigma\in \Sigma(\pi)$  by specifying whether the next entry is added to a new class of $\sigma$ or to an inactive class of $\sigma$ at each ambiguous vertex.  In this way, $\sigma$ is fully specified by $\pi$ and the location of the next entry of $\sigma$ for each ambiguous vertex of $\pi$.

\begin{ex}
\label{ex active classes}
Let $B=\overline{\Eq}_{4} \backslash Y$, where
\begin{equation*}
Y=\{\lvertb1\lvertm2 \h 4\lvertm3\lverte, \lvertb1\h 3\lvertm 2\h 4\lverte, \lvertb1 \h 3\lvertm2\lvertm4\lverte\}
\end{equation*}
and consider the avoidance set $\Av(B)$, whose factor graph $\Gamma_{B}$ is shown in Figure~\ref{ex 4.10 figure}.  Consider the path 
\begin{equation*}
\pi=\lvertb1\lvertm2 \h 4\lvertm 3\lverte \rightarrow \lvertb1 \h 3\lvertm 2 \h 4\lverte \rightarrow \lvertb1 \h 3\lvertm2\lvertm4\lverte
\end{equation*}
 in $\Gamma_{B}$.  Initially, the vertex $\lvertb1\lvertm2 \h 4\lvertm 3\lverte$ determines the classes of the first four entries of any element of $\Sigma(\pi)$.  Similarly, the second vertex dictates that the fifth entry of any element of $\Sigma(\pi)$ is in the same class as the third entry.  At this point, the class of the first entry is inactive.  The third vertex is ambiguous since it is a special vertex in a cycle; on entering it the next entry of an element of $\Sigma(\pi)$ can either be placed in the inactive class or in a new class.  This gives two relations in $\Sigma(\pi)$: $\lvertb1 \h 6\lvertm2 \h 4\lvertm3 \h 5\lverte$ and $\lvertb1\lvertm2 \h 4\lvertm3 \h 5\lvertm6\lverte$.  
\end{ex}
 
\begin{figure}
\begin{tikzpicture}[> =  {Stealth [scale=1.3]}, thick]
\tikzstyle{everystate} = [thick]
\node [state, shape = ellipse, minimum size = 20pt] (1b24b3) at (0,0) {\small{$\lvertb1\lvertm 2 \h 4\lvertm 3\lverte$}};
\node [state, shape = ellipse, minimum size = 20pt] (13b24) at (4, 0) {\small{$\lvertb1 \h 3\lvertm2 \h 4\lverte$}};
\node [state, shape = ellipse, minimum size = 20pt]  (13b2b4) at (8,0) {\small{$\lvertb1 \h 3\lvertm2\lvertm4\lverte$}};

\path[->]
(1b24b3) edge node {} (13b24)
(13b24) edge node {} (13b2b4)
(1b24b3) edge [bend left] node {} (13b2b4)
(13b2b4) edge [bend left] node {} (1b24b3)
(13b24) edge [loop below] node {} (13b24)
;
\end{tikzpicture}
\caption{The factor graph of $\Av(B)$ from Example \ref{ex active classes}.}
\label{ex 4.10 figure}
\end{figure}

\begin{lemma}
\label{lemma no amb verts in paths}
If $\sigma, \rho \in C_{[b, \infty)}$ and $\Pi(\sigma), \Pi(\rho)$ contain no ambiguous vertices then $\sigma \leq \rho$ if and only if $\Pi(\sigma) \leq \Pi(\rho)$.
\end{lemma}

\begin{proof}
($\Leftarrow$) Since there are no ambiguous vertices, $|\Sigma(\Pi(\sigma))|=|\Sigma(\Pi(\rho))|=1$, so $\Sigma(\Pi(\sigma))=\{\sigma\}$ and $\Sigma(\Pi(\rho))=\{\rho\}$ by Proposition \ref{prop properties of paths and relations}. Since each path only has one associated equivalence relation, $\sigma$ must be a sub-equivalence relation of $\rho$ as required.  ($\Rightarrow$) This is immediate from Proposition \ref{prop subrelns give rise to subpaths}.
\end{proof}

A consequence of Lemma \ref{lemma no amb verts in paths} is that there is a one-to-one correspondence between paths with no ambiguous vertices in $\Gamma_{B}$ and their associated equivalence relations in $C_{[b,\infty)}$.  The next lemma follows immediately from Lemma \ref{lemma no amb verts in paths}.

\begin{lemma}
\label{lemma no amb verts}
If a factor graph $\Gamma_{B}$ contains no ambiguous vertices and $\rho, \sigma \in C_{[b, \infty)}$ then $\sigma \leq \rho$ if and only if $\Pi(\sigma) \leq \Pi(\rho)$.\qed
\end{lemma}

So far, we have enough information to give the following partial version of our intended result, which considers the special case in which the factor graph of $C=\Av(B)$ contains no ambiguous vertices.  

\begin{prop}
\label{prop wqo no amb verts}
If the factor graph $\Gamma_{B}$ contains no ambiguous vertices then $C=\Av(B)$ is wqo if and only if $\Gamma_{B}$ is wqo.
\end{prop}

\begin{proof} 
By Lemma \ref{lemma no amb verts}, the poset of paths in $\Gamma_{B}$ is isomorphic to the poset of equivalence relations in $C_{[b, \infty)}$, and so $C_{[b, \infty)}$ is wqo if and only if $\Gamma_{B}$ is wqo.  Then by Lemma \ref{lemma Cb and C}, $C$ is wqo if and only if $\Gamma_{B}$ is wqo.
\end{proof}

\section{Two types of cycles which imply non-wqo}
\label{section wqo unbounded}

The purpose of this section is to show non-wqo for avoidance sets of equivalence relations whose factors graphs contain an in-out cycle or a special vertex in a cycle.  
We do this by utilising the relationship between the poset of paths in the factor graph and the poset of equivalence relations in an avoidance set explored in Section \ref{section factor graph}.  

The only ambiguous vertices we will need to consider are special vertices in cycles; we state the results in terms of special vertices, though it is their ambiguity (guaranteed by \ref{prop special vertex in cycle}) which is key to the outcome.  We will need to consider ambiguous vertices more generally in the following sections.

\begin{lemma}
\label{lemma not wqo in-out cycle}
If $\Gamma_{B}$ contains an in-out cycle then $C=\Av(B)$ is not wqo.
\end{lemma}

\begin{proof}
Since $\Gamma_{B}$ contains an in-out cycle, it is not wqo by Proposition \ref{prop digraphs wqo} so there is an infinite antichain of paths $\pi_{1}, \pi_{2}, \dots$ in $\Gamma_{B}$.  Aiming for a contradiction, suppose $C$ is wqo.  Take equivalence relations $\sigma_{i} \in \Sigma(\pi_{i})$ for $i=1, 2, \dots$ .  Since $C$ is wqo, $\sigma_{j} \leq \sigma_{k}$ for some $j \neq k$.  Then by Proposition \ref{prop subrelns give rise to subpaths} $\pi_{j} \leq \pi_{k}$, a contradiction.  We conclude that $C$ is not wqo.
\end{proof}

Now we turn our attention to avoidance sets whose factor graphs contain special vertices in cycles.

\begin{defn}
\label{defn unbounded class}
An avoidance set $C=\Av(B)$ is \emph{unbounded} if there is no (finite) upper bound on the number of equivalence classes  of its members; otherwise $C$ is \emph{bounded}. 
\end{defn}

\begin{lemma}
\label{lemma amb verts and bounds}
An avoidance set $C=\Av(B)$ is unbounded if and only if $\Gamma_{B}$ contains a cycle with a special vertex in it.
\end{lemma}

\begin{proof}
($\Rightarrow$) For any path in $\Gamma_{B}$, the only vertices where a class might be added to an associated equivalence relation are special vertices.  Since $\Gamma_{B}$ is a finite digraph, the only way to allow an unbounded number of classes in equivalence relations is for some path to visit a special vertex twice, i.e. if there is a special vertex in a cycle.

($\Leftarrow$) Suppose there is a special vertex $\nu_{a}$ in a cycle $\eta$ in $\Gamma_{B}$, and without loss of generality assume $\eta$ starts and ends at $\nu_{a}$.  Consider the equivalence relations $\theta_{k}\in \Sigma(\eta^{k})$, where $k\geq 1$, which add a new class each time an ambiguous vertex is entered (including $\nu_{a}$).  Each equivalence relation $\theta_{k}$ has at least $k$ classes.  Since this holds for any $k \geq 1$, $C$ is unbounded.
\end{proof}

\begin{lemma}
\label{lemma wqo and bounds}
If an avoidance set is unbounded then it is not wqo. 
\end{lemma}

\begin{proof}
Suppose $C=\Av(B)$ is unbounded.  By Lemma \ref{lemma amb verts and bounds}, $\Gamma_{B}$ contains a cycle $\xi$ with a special vertex $\nu_{a}$ in it, and $\nu_{a}$ is ambiguous by Proposition \ref{prop special vertex in cycle}.  Let $\mu$ be the vertex preceding $\nu_{a}$ in $\xi$.

For $k\geq3$, let $\pi_{k}$ be the path that starts at $\mu$, proceeds $k$ times around $\xi$ and ends at $\nu_{a}$.  We will look at the equivalence relations $\sigma_{k}$ in each $\Sigma(\pi_{k})$ such that for each $k$:
\begin{nitemize}
\item $\sigma_{k}$ has underlying set $[1, n^{(k)}]$.
\item $n_{1}^{k}$ is added to an inactive class of $\sigma_{k}$ the second time $\nu_{a}$ is entered.
\item $n_{2}^{k}$ is added to an inactive class of $\sigma_{k}$ the last time $\nu_{a}$ is entered.
\item at all other visits to special vertices, a new class is added to $\sigma_{k}$.
\end{nitemize}

We claim that the set $\{\sigma_{k}: k\geq 3\}$ forms an infinite antichain. 

Aiming for a contradiction, suppose that $\sigma_{i} \leq \sigma_{j}$ for some $j >i\geq3$.  Suppose $f: [1, n^{(i)}] \rightarrow [1, n^{(j)}]$ is the underlying embedding.  It can be seen that $n_{1}^{i}, n_{2}^{i}, n_{1}^{j}, n_{2}^{j}$ are the only entries added on entering $\nu_{a}$ which are not the smallest element of their classes.  Therefore $f$ must map $n_{1}^{i}$  to $n_{1}^{j}$, and since $i<j$ this forces $f$ to map $n_{2}^{i}$ to an element of $[1, n^{(j)}]$ which is the smallest element in its class.  This is a contradiction, since $n_{2}^{i}$ is not the smallest element of its class, so this prevents $f$ from preserving equivalence classes.  Therefore $\sigma_{i} \nleq \sigma_{j}$, so $\{\sigma_{k}: k\geq 3\}$ is an infinite antichain and $C$ is not wqo. 
\end{proof}

\section{Coloured Equivalence Relations}
\label{section coloured equivs}

We have dealt with factor graphs where there is a cycle containing a special vertex or there is an in-out cycle in Section \ref{section wqo unbounded}.  Now we look at the remaining avoidance sets.  Since the factor graph of any such avoidance set has no special vertices in cycles, there is a bound on the number of classes members of the avoidance set may have.  This motivates the concept of \emph{coloured equivalence relations} and the \emph{coloured factor graphs} associated with them; the idea being that we can `encode' equivalence classes of members of a bounded avoidance set using only a finite amount of additional information.

In this section we introduce these concepts and then explore the relationship between the coloured and uncoloured 
versions.  Unlike the uncoloured case, there will be a one-to-one correspondence between coloured equivalence relations and paths in their coloured factor graphs, bypassing the multiple choices previously arising at ambiguous vertices.  This will enable us to tackle the wqo question for the remaining avoidance sets in Section \ref{section wqo in general}.

\begin{defn}
\label{defn colouring equivs}
For $k\geq 1$, a \emph{$k$-colouring} of an equivalence relation $\sigma$ is an injective mapping from the set of equivalence classes of $\sigma$ to $[k]$.  In this context we call the elements of $[k]$ \emph{colours}.  An equivalence relation $\sigma$ together with a $k$-colouring is called a \emph{$k$-coloured equivalence relation}. 
\end{defn}

When the value of $k$ is not important, we will speak of \emph{colourings} and \emph{coloured} equivalence relations.  An equivalence relation without a colouring will be called an \emph{uncoloured} equivalence relation.  We will distinguish between coloured and uncoloured equivalence relations by writing coloured equivalence relations with their colourings as superscripts.  For example, if $\sigma$ is an uncoloured relation and $c$ is a colouring of $\sigma$, the coloured equivalence relation of $\sigma$ with $c$ will be written $\sigma^{c}$.  In concrete examples we will underline the equivalence classes and put their colours as subscripts, e.g. see Example \ref{coloured relns}.

\begin{defn}
\label{defn coloured cons order}
Let $\sigma^{c_1}$, $\rho^{c_2}$ be two $k$-coloured equivalence relations.
We say that $\sigma^{c_1}\leq_{\col} \rho^{\sigma_2}$ if there exists a contiguous embedding $f$ of $\sigma$ into $\rho$ which respects colourings; more specifically, for every equivalence class $C$ of $\sigma$, we require $c_1(C)=c_2(D)$, where $D$ is the unique equivalence class of $\rho$ such that $f(C)\subseteq D$.
We also say that $\sigma^{c_1}$ is a \emph{coloured sub-equivalence relation} of $\rho^{c_{2}}$.
We call $\leq_{\col}$ the \emph{coloured consecutive embedding order}.

\end{defn}

From now on we will denote $\leq_{\col}$ simply by $\leq$, since it is always clear from the nature of the equivalence relations which order is meant.  Also note that $\sigma^{c_1}\leq \rho^{c_2}$ implies $\sigma\leq\rho$.

Given a coloured equivalence relation $\sigma^{c}$ on a set $X$ and any subset $Y$ of $X$, the restriction of $\sigma^{c}$ to points in $Y$ is denoted $\sigma^{c} \Harpoon{Y}$.  As with uncoloured equivalence relations, $\sigma^{c} \Harpoon{Y}$ is a coloured sub-equivalence relation of $\sigma^{c}$ and any coloured sub-equivalence relation of $\sigma^{c}$ can be expressed as a restriction of $\sigma^{c}$ to a subset of $X$.

\begin{defn}
\label{defn col iso}
Two coloured equivalence relations $\sigma^{c_{1}}, \rho^{c_{2}}$ are \emph{isomorphic} if there exists a contiguous bijection from $\sigma^{c_{1}}$ to $\rho^{c_{2}}$ that preserves equivalence classes and colourings.
\end{defn}

We take $\overline{\Eq}^{\h \col}$ to be the set of finite coloured equivalence relations (modulo isomorphism), and consider the poset $(\overline{\Eq}^{\h \col},\leq_{\col})$.

\begin{ex}
\label{coloured relns}
Let $\sigma=\lvertb 1 \lvertm 2 \h 4 \lvertm 3\lverte$ and $\rho=\lvertb 1 \h 2 \lvertm 3 \h 8 \lvertm 4 \h 6\lvertm 5\lvertm7\lverte$ be uncoloured equivalence relations.  It is easy to see that $\sigma \leq \rho$.  Now consider the colourings
\begin{align*}
\sigma^{c_{1}}&=\lvertb \underline{\colorbox{blue!30}{1}}_{\hh1} \hspace{-1mm}\lvertm \underline{\colorbox{orange}{2 4}}_{\hh2} \hspace{-0.75mm}\lvertm \underline{\colorbox{green}{3}}_{\hh3} \hspace{-0.9mm}\lverte\\
\sigma^{c_{2}}&=\lvertb \underline{\colorbox{blue!30}{1}}_{\hh1} \hspace{-1mm}\lvertm \underline{\colorbox{pink}{2 4}}_{\hh4} \hspace{-1mm}\lvertm \underline{\colorbox{yellow}{3}}_{\hh5} \hspace{-1mm}\lverte\\
\rho^{c_{3}}&=\lvertb \underline{\colorbox{pink}{1 2}}_{\hh 4} \hspace{-1mm} \lvertm \underline{\colorbox{blue!30}{3 8}}_{\hh1} \hspace{-1mm} \lvertm \underline{\colorbox{orange}{4 6}}_{\hh 2} \hspace{-1mm}\lvertm \underline{\colorbox{green}{5}}_{\hh 3} \hspace{-1mm}\lvertm \underline{\colorbox{purple!20}{7}}_{\hh 6} \hspace{-1mm}\lverte.
\end{align*}

It can be seen that $\sigma^{c_{1}}\leq \rho^{c_{3}}$ since the contiguous map $f:[4] \rightarrow [7]$ defined by $f(1)=3$ preserves both equivalence classes and colourings. 

On the other hand, $\sigma^{c_{2}}\nleq \rho^{c_{3}}$ since it is not possible to map the class $\lvertm \underline{\colorbox{pink}{2 4}}_{\hh4} \hspace{-1mm}\lvertm$ of $\sigma^{c_{2}}$ to the class $\lvertb \underline{\colorbox{pink}{1 2}}_{\hh 4} \hspace{-1mm} \lverte$ of $\rho^{c_{3}}$ contiguously.
\end{ex}

Suppose $C=\Av(B)$ is a bounded avoidance set of equivalence relations, with bound $k$ on the number of equivalence classes of its elements.  We define $C^{\col}$ to be the set of $k$-coloured elements of $C$.  If $X \subseteq \mathbb{N}$, we take $C_{X}^{\col}= \{\sigma^{c} \in C^{\col} : |\sigma| \in X\}$.

Let $C$ and $k$ be as above.
Let $\overline{\Eq}^{\h \col}_{b}$ denote the set of all $k$-coloured equivalence relations on the set $[b]$.  
We take $G_{b}^{\col}$ to be the digraph with vertex set $\overline{\Eq}_{b}^{\h\col}$ and an edge from vertex $\mu^{c_{1}}$ to vertex $\nu^{c_{2}}$ if and only if $\mu^{c_{1}}\Harpoon{[2, b]} \cong \nu^{c_{2}}\Harpoon{[1, b-1]}$.  
We define the \emph{coloured equivalence relation factor graph} $\Gamma_{B}^{\col}$ of $C$ as the induced subgraph of $G_{b}^{\col}$ with vertex set $C_{b}^{\col}=\overline{\Eq}_{b}^{\h\col} \cap C^{\col}$; this will be referred to as the \emph{coloured factor graph} from now on.

\begin{ex}
\label{ex col graph}
Let $B=\overline{\Eq}_{4} \backslash X$ where
\begin{equation*}
X= \{\lvertb1\lvertm 2 \h 3 \h 4\lverte, \lvertb 1 \h 2 \h 3 \lvertm 4\lverte\}
\end{equation*}
and consider $C=\Av(B)$.  The factor graph $\Gamma_{B}$ is shown in Figure \ref{ex 7.9 figure}.  It is easy to see that $\lvertb 1 \h 2 \h 3\lvertm4\lverte$ is the only ambiguous vertex and that the bound on the number of classes of elements of $C$ is $k=3$.  The vertices of the coloured factor graph $\Gamma_{B}^{\col}$ of $C$ are the 3-colourings of vertices of $\Gamma_{B}$.  The coloured factor graph is also shown in Figure \ref{ex 7.9 figure}.  There is an edge from vertex $\lvertb \underline{\colorbox{blue!30}{1}}_{\hh1} \hspace{-0.8mm}\lvertm \underline{\colorbox{orange}{2 3 4}}_{\hh2} \hspace{-0.8mm}\lverte$ to vertex $\lvertb\cequiv{orange}{1 2 3}{2}\lvertm\cequiv{blue!30}{4}{1}\lverte$ because 
\begin{equation*}
\lvertb \underline{\colorbox{blue!30}{1}}_{\hh1} \hspace{-0.8mm}\lvertm \underline{\colorbox{orange}{2 3 4}}_{\hh2} \hspace{-0.8mm}\lverte \Harpoonc{\{2, 3, 4\}} = \lvertb \underline{\colorbox{orange}{2 3 4}}_{\hh2} \hspace{-0.8mm}\lverte \cong \lvertb\cequiv{orange}{1 2 3}{2}\lverte = \lvertb\cequiv{orange}{1 2 3}{2}\lvertm\cequiv{blue!30}{4}{1}\lverte \Harpoonc{\{1, 2, 3\}}.
\end{equation*}

As there are no cycles in $\Gamma_{B}$, $C$ is finite and so $C^{\col}$ is also finite.  The elements of $C^{\col}$ are all 3-colourings of equivalence relations in $C$; in other words, all 3-colourings of the equivalence relations on $<4$ points and of $\lvertb1\lvertm 2\h 3\h4\lverte$, $\lvertb1 \h 5\lvertm 2 \h3\h4\lverte$, $\lvertb1\lvertm2\h3\h4\h\lvertm5\lverte$ and $\lvertb1\h2\h3\lvertm4\lverte$.
\end{ex}

\begin{figure}
\begin{center}
\begin{tikzpicture}[> =  {Stealth [scale=1.3]}, thick]
\tikzstyle{everystate} = [thick]
\node [state, shape = ellipse, minimum size = 20pt] (v1) at (2,0) {\small{$\lvertb1\lvertm 2 \h 3 \h 4\lverte$}};
\node [state, shape = ellipse, minimum size = 20pt] (v2) at (6,0) {\small{$\lvertb1 \h 2 \h 3 \lvertm4\lverte$}};

\path[->]
(v1) edge node {} (v2)
;
\end{tikzpicture}\vspace{10mm}\\
\begin{tikzpicture}[> =  {Stealth [scale=1.3]}, thick]
\tikzstyle{everystate} = [thick]
\node [state, shape = ellipse, minimum size = 20pt] (bo) at (1,0) {\small{$\lvertb \underline{\colorbox{blue!30}{1}}_{\hh1} \hspace{-0.8mm}\lvertm \underline{\colorbox{orange}{2 3 4}}_{\hh2} \hspace{-0.8mm}\lverte$}};
\node [state, shape = ellipse, minimum size = 20pt] (po) at (1,-1.5) {\small{$\lvertb \underline{\colorbox{pink}{1}}_{\hh 3}\hspace{-0.8mm}\lvertm\underline{\colorbox{orange}{2 3 4}}_{\hh 2} \hspace{-0.8mm}\lverte$}};
\node [state, shape = ellipse, minimum size = 20pt] (bp) at (1,-3) {\small{$\lvertb \cequiv{blue!30}{1}{1}\lvertm\cequiv{pink}{2 3 4}{3}\lverte$}};
\node [state, shape = ellipse, minimum size = 20pt] (op) at (1,-4.5) {\small{$\lvertb\cequiv{orange}{1}{2}\lvertm\cequiv{pink}{2 3 4}{3}\lverte$}};
\node [state, shape = ellipse, minimum size = 20pt] (ob) at (1,-6) {\small{$\lvertb\cequiv{orange}{1}{2}\lvertm\cequiv{blue!30}{2 3 4}{1}\lverte$}};
\node [state, shape = ellipse, minimum size = 20pt] (pb) at (1,-7.5) {\small{$\lvertb\cequiv{pink}{1}{3}\lvertm\cequiv{blue!30}{2 3 4}{1}\lverte$}};

\node [state, shape = ellipse, minimum size = 20pt] (ob2) at (7,0) {\small{$\lvertb\cequiv{orange}{1 2 3}{2}\lvertm\cequiv{blue!30}{4}{1}\lverte$}};
\node [state, shape = ellipse, minimum size = 20pt] (op2) at (7,-1.5) {\small{$\lvertb\cequiv{orange}{1 2 3}{2}\lvertm\cequiv{pink}{4}{3}\lverte$}};
\node [state, shape = ellipse, minimum size = 20pt] (pb2) at (7,-3) {\small{$\lvertb\cequiv{pink}{1 2 3}{3}\lvertm\cequiv{blue!30}{4}{1}\lverte$}};
\node [state, shape = ellipse, minimum size = 20pt] (po2) at (7,-4.5) {\small{$\lvertb\cequiv{pink}{1 2 3}{3}\lvertm\cequiv{orange}{4}{2}\lverte$}};
\node [state, shape = ellipse, minimum size = 20pt] (bo2) at (7,-6) {\small{$\lvertb\cequiv{blue!30}{1 2 3}{1}\lvertm\cequiv{orange}{4}{2}\lverte$}};
\node [state, shape = ellipse, minimum size = 20pt] (bp2) at (7,-7.5) {\small{$\lvertb\cequiv{blue!30}{1 2 3}{1}\lvertm\cequiv{pink}{4}{3}\lverte$}};

\path[->]
(bo) edge node {} (ob2)
(bo) edge node {} (op2)
(po) edge node {} (ob2)
(po) edge node {} (op2)
(bp) edge node {} (pb2)
(bp) edge node {} (po2)
(op) edge node {} (pb2)
(op) edge node {} (po2)
(ob) edge node {} (bo2)
(ob) edge node {} (bp2)
(pb) edge node {} (bo2)
(pb) edge node {} (bp2)
;
\end{tikzpicture}
\caption{The uncoloured and coloured factor graphs of $\Av(B)$ from Example \ref{ex col graph}.}
\label{ex 7.9 figure}
\end{center}
\end{figure}

Let $\sigma^{c} \in C_{[b, \infty)}^{\col}$ be a coloured equivalence relation, and without loss of generality assume its underlying set is $[n]$ for some $n\in \mathbb{N}$.  We associate $\sigma^{c}$ with the path $\Pi^{\prime}(\sigma^{c})$ given by
\begin{equation*}
\sigma^{c}\Harpoon{[1, b]} \rightarrow \sigma^{c}\Harpoon{[2, b+1]}\downarrow \rightarrow \dots \rightarrow \sigma^{c}\Harpoon{[n-b+1, n]}\downarrow
\end{equation*}

in $\Gamma_{B}^{\col}$.  
The notation $\rho^{c}\downarrow$ means the unique equivalence relation on $[|\rho|]$ isomorphic to $\rho^{c}$, in line with uncoloured relations, as introduced in Section \ref{section factor graph}.

On the other hand, we associate a path $\pi$ in $\Gamma_{B}^{\col}$ with the coloured equivalence relation $\Sigma^{\prime}(\pi) \in C_{[b, \infty)}^{\col}$ such that $\Pi^{\prime}(\Sigma^{\prime}(\pi))=\pi$.  Note that, unlike the analogue for uncoloured equivalence relations, $\Sigma^{\prime}(\pi)$ will always be a single coloured equivalence relation.

\begin{ex}
\label{ex paths and relns}
Consider again the avoidance set from Example \ref{ex col graph}, $C=\Av(B)$  for $B=\overline{\Eq}_{4} \backslash X$ and 
\begin{equation*}
X= \{\lvertb1\lvertm 2 \h 3 \h 4\lverte, \lvertb 1 \h 2 \h 3 \lvertm 4\lverte\}.
\end{equation*}
The coloured factor graph of $C$ is shown in Figure \ref{ex 7.9 figure}.  Here the path 
\begin{equation*}
\lvertb\cequiv{pink}{1}{3}\lvertm\cequiv{blue!30}{2 3 4}{1}\lverte \rightarrow \lvertb\cequiv{blue!30}{1 2 3}{1}\lvertm\cequiv{pink}{4}{3}\lverte
\end{equation*}
is associated with the coloured equivalence relation $\lvertb \cequiv{pink}{1 5}{3}\lvertm \cequiv{blue!30}{2 3 4}{1}\lverte$.  

On the other hand, the coloured equivalence relation $\lvertb \cequiv{blue!30}{1}{1}\lvertm\cequiv{orange}{2 3 4}{2}\lvertm \cequiv{pink}{5}{3}\lverte$ is associated with the path 
\begin{equation*}
\lvertb \cequiv{blue!30}{1}{1}\lvertm\cequiv{orange}{2 3 4}{2}\lverte \rightarrow \lvertb\cequiv{orange}{1 2 3}{2}\lvertm \cequiv{pink}{4}{3}\lverte.
\end{equation*}
\end{ex}

It can be seen that $\Pi^{\prime}$ and $\Sigma^{\prime}$ are mutual inverses, as stated in the following proposition.

\begin{prop}
\label{prop basic coloured properties}
\begin{thmenumerate}
\item If $\sigma^{c} \in C_{[b, \infty)}^{\col}$, then $\sigma^{c} = \Sigma^{\prime}(\Pi^{\prime}(\sigma^{c}))$.
\item If $\pi$ is a path in $\Gamma_{B}^{\col}$, then $\pi=\Pi^{\prime}(\Sigma^{\prime}(\pi))$.
\end{thmenumerate}
\qed
\end{prop}

In this way, there is bijective correspondence between coloured equivalence relations and their associated paths.  Moreover, this correspondence respects the coloured consecutive ordering in the following sense.

\begin{prop}
\label{prop col respects ordering}
If $\sigma^{c_{1}}, \rho^{c_{2}} \in C_{[b, \infty)}^{\col}$ then $\sigma^{c_{1}} \leq \rho^{c_{2}}$ if and only if $\Pi^{\prime}(\sigma^{c_{1}}) \leq \Pi^{\prime}(\rho^{c_{2}})$ in $\Gamma_{B}^{\col}$.
\end{prop}

\begin{proof}
($\Rightarrow$) This is analogous to the uncoloured version (Proposition \ref{prop subrelns give rise to subpaths}).

($\Leftarrow$) Since $\Pi^{\prime}(\sigma^{c_{1}}) \leq \Pi^{\prime}(\rho^{c_{2}})$ and each path is associated with a single coloured equivalence relation, $\sigma^{c_{1}}=\Sigma^{\prime}(\Pi^{\prime}(\sigma^{c_{1}})) \leq \Sigma^{\prime}(\Pi^{\prime}(\rho^{c_{2}}))=\rho^{c_{2}}$.
\end{proof}

We note in passing that decidability of the wqo and atomicity problems for avoidance sets of the poset of $k$-coloured equivalence relations under the coloured consecutive embedding order is an immediate consequence of Proposition \ref{prop col respects ordering}

\section{WQO under the Consecutive Embedding Order}
\label{section wqo in general}

In this section we establish decidability of wqo for avoidance sets of equivalence relations under the consecutive embedding order.    We use the tools introduced in Section \ref{section coloured equivs} to relate wqo in factor graphs to wqo in coloured factor graphs.  This will allow us to show that in the remaining cases all avoidance sets are wqo.  We finish the section by combining this with the results of Section \ref{section wqo unbounded} to show decidability of wqo for avoidance sets of equivalence relations under the consecutive embedding order in general.

\begin{lemma}
\label{lemma collection graphs and new graph wqo}
If $\Gamma_{B}$ contains no special vertices in cycles and is wqo, then $\Gamma_{B}^{\col}$ is also wqo. 
\end{lemma}

\begin{proof}
Since $\Gamma_{B}$ is wqo, it has no in-out cycles by Proposition \ref{prop digraphs wqo}.  Aiming for a contradiction, suppose that $\Gamma_{B}^{\col}$ is not wqo, so has an in-out cycle $\bar{\eta}$.  Let $\bar{\eta}$ have in-edge $(\mu^{c_{1}}, \nu^{c_{2}})$ and out-edge $(\sigma^{c_{3}}, \rho^{c_{4}})$. 

The cycle $\bar{\eta}$ must correspond to a cycle $\eta$ in $\Gamma_{B}$, and by assumption $\eta$ contains no special vertices, so neither does $\bar{\eta}$.  Since $\Gamma_{B}$ is wqo, $\eta$ is not an in-out cycle by Proposition \ref{prop digraphs wqo}.  
We split considerations into two cases.

\textit{Case 1: $\eta$ is not an out-cycle.}
The assumption implies that $(\sigma, \rho)$ is the only edge starting at $\sigma$ in $\Gamma_{B}$.  On the other hand, in $\Gamma_{B}^{\col}$ there are at least two edges starting at $\sigma^{c_{3}}$: an edge in $\bar{\eta}$ and the out-edge $(\sigma^{c_{3}}, \rho^{c_{4}})$.  Both of these edges in $\Gamma_{B}^{\col}$ correspond to the edge $(\sigma, \rho)$ in $\Gamma_{B}$.  This means that there is another colouring $c_{5}$ of $\rho$ such that $(\sigma^{c_{3}}, \rho^{c_{5}})$ is an edge in $\bar{\eta}$.  Since $(\sigma^{c_{3}}, \rho^{c_{4}})$, $(\sigma^{c_{3}}, \rho^{c_{5}})$ are edges, we have that $\rho^{c_{4}}\Harpoon{[1, b-1]} \cong \sigma^{c_{3}}\Harpoon{[2, b]} \cong \rho^{c_{5}}\Harpoon{[1, b-1]}$.  Therefore, $\rho^{c_{4}}\Harpoon{[1, b-1]} = \rho^{c_{5}}\Harpoon{[1, b-1]}$.  Since $c_{4}$ and $c_{5}$ are distinct colourings, $b$ is coloured differently under each of these.  If $|\rho_{b}|\neq1$, the colour of $b$ would be uniquely determined by the colour of other elements in its class, so $c_{4}$ and $c_{5}$ would be identical.  Therefore, it must be the case that $|\rho_{b}|=1$, so $\rho$ is a special vertex in 
the cycle $\eta$, a contradiction.  

\textit{Case 2: $\eta$ is not an in-cycle.}
Now $(\mu, \nu)$ is the only edge ending at $\nu$ in $\Gamma_{B}$.  However, there are at least two edges ending at $\mu^{c_{1}}$ in $\Gamma_{B}^{\col}$: an edge in $\bar{\eta}$ and the in-edge $(\mu^{c_{1}}, \nu^{c_{2}})$.  Both of these must correspond to $(\mu, \nu)$  in $\Gamma_{B}$, so there is another colouring $c_{6}$ of $\mu$ such that $(\mu^{c_{6}}, \nu^{c_{2}})$ is the edge in $\bar{\eta}$.  Then since $(\mu^{c_{1}}, \nu^{c_{2}})$, $(\mu^{c_{6}}, \nu^{c_{2}})$ are edges, $\mu^{c_{1}}\Harpoon{[2, b]} \cong \nu^{c_{2}}\Harpoon{[1, b-1]} \cong \mu^{c_{6}}\Harpoon{[2, b]}$, meaning that $\mu^{c_{1}}\Harpoon{[2, b]} = \mu^{c_{6}}\Harpoon{[2, b]}$.  If $|\mu_{1}|\neq 1$, the colour of $1$ is determined by the colour of other elements in its class, so $c_{1}$ and $c_{6}$ are not distinct.  Therefore, $|\mu_{1}|= 1$, and from this we can see that $\mu$ has one more class than $\nu$.  Now consider traversing the subpath of $\eta$ from $\nu$ to $\mu$.  Since $\mu$ has one more class than $\nu$, at some point in this path there is an edge $(\nu^{\prime}, \mu^{\prime})$ such that $\nu^{\prime}$ has fewer classes than $\mu^{\prime}$; the only way for this to happen is for $\mu^{\prime}$ to be special, a contradiction.

In each of the two cases we obtained a contradiction, and this completes the proof.
\end{proof}

\begin{lemma}
\label{lemma collection graph wqo so is C}
If the factor graph $\Gamma_{B}$ has no special vertices in cycles and is wqo, then the avoidance set $C=\Av(B)$ is also wqo. 
\end{lemma}

\begin{proof}
Aiming for a contradiction, suppose that $\Gamma_{B}$ is wqo but $C_{[b, \infty)}$ is not.  Take an antichain $\sigma_{1}, \sigma_{2}, ... \in C_{[b, \infty)}$.  Since $\Gamma_{B}$ is wqo, so is $\Gamma_{B}^{\col}$ by Lemma \ref{lemma collection graphs and new graph wqo}, meaning that for some $i, j$ and colourings $c_{k}, c_{l}$ of $\sigma_{i}, \sigma_{j}$ respectively, we have that $\Pi^{\prime}(\sigma_{i}^{c_{k}}) \leq \Pi^{\prime}(\sigma_{j}^{c_{l}})$.  Then by Proposition \ref{prop col respects ordering}, $\sigma_{i}^{c_{k}} \leq \sigma_{j}^{c_{l}}$ and hence $\sigma_{i} \leq \sigma_{j}$, a contradiction.  Therefore if $\Gamma_{B}$ is wqo, so is $C_{[b, \infty)}$ and then by Lemma \ref{lemma Cb and C}, $C$ is wqo as required.
\end{proof}

We can now prove our main results concerning the wqo question for the consecutive order:

\begin{thrm}
\label{thrm wqo cdn cons}
A finitely based avoidance set $C=\Av(B)$ is wqo under the consecutive embedding ordering if and only if $\Gamma_{B}$ has no in-out cycles and no special vertices in cycles.
\end{thrm}

\begin{proof}
($\Rightarrow$) Suppose $C=\Av(B)$ is wqo.  By Lemma \ref{lemma wqo and bounds} and Lemma \ref{lemma amb verts and bounds}, $\Gamma_{B}$ cannot contain any special vertices in cycles.  By Lemma \ref{lemma not wqo in-out cycle}, $\Gamma_{B}$ cannot contain any in-out cycles.\\ 
($\Leftarrow$) Suppose $\Gamma_{B}$ contains no in-out cycles or special vertices in cycles.  Since $\Gamma_{B}$ contains no in-out cycles, it is wqo by Proposition \ref{prop digraphs wqo} and therefore we can apply Lemma \ref{lemma collection graph wqo so is C} to see that $C$ is wqo, completing the proof.
\end{proof}

\begin{thrm}
\label{thrm decidability wqo cons}
It is decidable whether a finitely based avoidance set $C=\Av(B)$ is wqo under the consecutive embedding ordering.
\end{thrm}

\begin{proof}
It is decidable whether $\Gamma_{B}$ is wqo by Proposition \ref{prop digraphs wqo} since we can check for in-out cycles.  It is also decidable whether $\Gamma_{B}$ contains special vertices in cycles. Therefore the conditions of Theorem \ref{thrm wqo cdn cons} are decidable and so the result follows.
\end{proof}

\begin{ex}
\label{ex section 4 summary 2}
Let $B=\{\lvertb 1 \h 2 \h 3\lverte, \lvertb 1 \h3\lvertm 2 \lverte\}$ and consider $C=\Av(B)$.  Figure \ref{ex 4.27 figure} shows the factor graph $\Gamma_{B}$. Since $\lvertb1\lvertm2\lvertm3\lverte$ is a special vertex in a cycle, $C$ is not wqo by Theorem \ref{thrm wqo cdn cons}.  In addition, the path 
\begin{equation*}
\lvertb1\h 2 \lvertm 3 \lverte \rightarrow \lvertb 1 \lvertm2\lvertm3\lverte \rightarrow \lvertb 1 \lvertm2\lvertm3\lverte \rightarrow \lvertb1\lvertm2 \h 3\lverte
\end{equation*}
 forms an in-out cycle, breaking the other required condition for wqo.  An example of an infinite antichain in $C$ is $\{ \lvertb 1 \h n \lvertm 2\lvertm \dots \lvertm n-1\lverte : n\geq 4\}$.
 \end{ex}

\begin{figure}
\begin{tikzpicture}[> =  {Stealth [scale=1.3]}, thick]
\tikzstyle{everystate} = [thick]
\node [state, shape = ellipse, minimum size = 20pt] (1b2b3) at (8,0) {\small{$\lvertb1\lvertm2\lvertm3\lverte$}};
\node [state, shape = ellipse, minimum size = 20pt] (1b23) at (0, 0) {\small{$\lvertb1\lvertm2 \h 3\lverte$}};
\node [state, shape = ellipse, minimum size = 20pt]  (12b3) at (4,0) {\small{$\lvertb1 \h 2\lvertm3\lverte$}};

\path[->]
(1b23) edge [bend left] node {} (12b3)
(12b3) edge node {} (1b23)
(12b3) edge node {} (1b2b3)
(1b2b3) edge [loop right] node {} (1b2b3)
(1b2b3) edge [bend left] node {} (1b23)
;
\end{tikzpicture}
\caption{The factor graph of $\Av(\lvertb 1 \h 2 \h3\lverte, \lvertb 1 \h 3\lvertm 2 \lverte)$.}
\label{ex 4.27 figure}
\end{figure}

\begin{ex}
\label{ex section 4 summary 3}
Let $B=\{\lvertb1\lvertm2\lvertm3\lverte, \lvertb1 \h 2\lvertm3\lverte\}$ and $C=\Av(B)$.  The factor graph $\Gamma_{B}$ can been seen in Figure \ref{ex 4.28 figure}. This graph is wqo and contains no special vertices, so certainly has none in cycles, meaning that $C$ is wqo. 
\end{ex}

\begin{figure}
\begin{tikzpicture}[> =  {Stealth [scale=1.3]}, thick]
\tikzstyle{everystate} = [thick]
\node [state, shape = ellipse, minimum size = 20pt] (123) at (7,0) {\small{$\lvertb1 \h 2 \h 3\lverte$}};
\node [state, shape = ellipse, minimum size = 20pt] (13b2) at (0, 0) {\small{$\lvertb1 \h 3\lvertm2\lverte$}};
\node [state, shape = ellipse, minimum size = 20pt]  (1b23) at (3.5,0) {\small{$\lvertb1\lvertm2 \h 3\lverte$}};

\path[->]
(13b2) edge node {} (1b23)
(1b23) edge node {} (123)
(123) edge [loop right] node {} (123)
(13b2) edge [loop left] node {} (13b2)
;
\end{tikzpicture}
\caption{The factor graph of $\Av(\lvertb1\lvertm2\lvertm3\lverte, \lvertb1 \h 2\lvertm3\lverte)$.}
\label{ex 4.28 figure}
\end{figure}

\section{Atomicity under the consecutive embedding ordering}
\label{section atomicity cons}

In this section we show decidability of atomicity for avoidance sets of the poset of equivalence relations under the consecutive embedding order.  To do this we use the relationship between the poset of paths in the factor graph and the poset $C_{[b, \infty)}$ discussed in Section \ref{section factor graph}.
Unless otherwise specified, $B$ is understood to be an arbitrary finite set of relations, and $C=\Av(B)$.
In one direction the connection is straightforward:

\begin{lemma}
\label{lemma collection graph not atomic so is C}
If $C=\Av(B)$  is atomic, then $\Gamma_{B}$ is atomic.
\end{lemma}

\begin{proof}
Aiming for a contradiction, suppose that $\Gamma_{B}$ is not atomic.  Take two paths $\pi, \eta$ in $\Gamma_{B}$ such that there is no path containing both $\pi$ and $\eta$, and let $\alpha \in \Sigma(\pi)$ and $\beta \in \Sigma(\eta)$.  Since $C$ is atomic, there is an equivalence relation $\theta \in C$ which contains both $\alpha$ and $\beta$.  Since $b \leq|\alpha| \leq |\theta|$, $\Pi(\theta)$ is a path in $\Gamma_{B}$.  By Proposition \ref{prop subrelns give rise to subpaths}, $\Pi(\alpha), \Pi(\beta) \leq \Pi(\theta)$, in other words $\pi, \eta \leq \Pi(\theta)$.  This is a contradiction, and $\Gamma_B$ is  atomic. 
\end{proof}

In the reverse direction, the situation is complicated by the fact that a $\Gamma_B$ may be atomic for two different reasons -- if it is strongly connected or a bicycle (see Proposition \ref{prop digraphs atomic}) -- and also because in general atomicity of $\Gamma_B$ is not sufficient for that of $C$.

\begin{lemma}
\label{lemma strongly connected atomic}  
If $\Gamma_{B}$ is strongly connected then $C_{[b, \infty)}$ is atomic.
\end{lemma}

\begin{proof}
Take $\alpha, \beta \in C_{[b, \infty)}$ and let $\eta,\zeta$ be paths such that $\Pi(\alpha)=\eta$ and $\Pi(\beta)=\zeta$.  Since $\Gamma_{B}$ is strongly connected there is a path $\xi$ from the end vertex of $\eta$ to the start vertex of $\zeta$.  Let $\pi=\eta\xi\zeta$.  Consider the equivalence relation $\theta \in \Sigma(\pi)$ which is formed by making the same choices at ambiguous vertices as $\alpha$ when $\eta$ is traversed and the same choices as $\beta$ when $\zeta$ is traversed.  This means that $\alpha, \beta \leq \theta$, hence $C_{[b, \infty)}$ satisfies the JEP and so is atomic by Proposition \ref{prop jep}.
\end{proof}

\begin{lemma}
\label{lemma atomicity no amb verts}
Suppose $\Gamma_{B}$ contains no ambiguous vertices.  
Then $C_{[b, \infty)}$ is atomic if and only if $\Gamma_{B}$ is atomic.
\end{lemma}

\begin{proof}
If $\Gamma_{B}$ contains no ambiguous vertices, we saw in Lemma \ref{lemma no amb verts} that the poset of paths in $\Gamma_{B}$ is isomorphic to the poset of equivalence relations in $C_{[b, \infty)}$.  It follows that $\Gamma_{B}$ is atomic if and only if $C_{[b, \infty)}$ is atomic.
\end{proof}

\begin{lemma}
\label{lemma amb vert not in cycle then not atomic}
If $\Gamma_{B}$ contains an ambiguous vertex which is not in a cycle, then $C=\Av(B)$ is not atomic.
\end{lemma}

\begin{proof}
Let $\nu$ be such a vertex. Then there is an edge $(\mu, \nu)$, and a path $\pi$ ending at $\mu$, such that there exists an equivalence relation $\sigma\in\Sigma(\pi)$ with an inactive class.  Let $\xi$ be the concatenation of $\pi$ and the edge  $(\mu,\nu)$. 
Then we can take $\sigma_{1} \in \Sigma(\xi)$ to add a new class to $\sigma$ at $\nu$, and $\sigma_{2} \in \Sigma(\xi)$ to add the new entry to an inactive class at $\nu$.  
Suppose that there is an equivalence relation $\theta \in C$ containing both $\sigma_{1}$ and $\sigma_{2}$.  Note that $\theta \in C_{[b, \infty)}$ since $|\sigma_{1}|,|\sigma_2|\geq b$.  
As $\nu$ is not in a cycle, $\Pi(\theta)$ can only enter it once.  Then since both $\sigma_{1}$ and $\sigma_{2}$ are consecutive sub-equivalence relations of $\theta$, in $\theta$ we must both add a new class and add to an inactive class at $\nu$, a contradiction.  Therefore $C$ does not satisfy the JEP, meaning that it is not atomic by Proposition \ref{prop jep}. 
\end{proof}

\begin{lemma}
\label{lemma rest not atomic}
If $\Gamma_{B}$ is a bicycle, but not a cycle, with an ambiguous vertex then $C=\Av(B)$ is not atomic.
\end{lemma}

\begin{proof}
Let $\nu$ be an ambiguous vertex. If $\nu$ is in neither the initial nor terminal cycle, this is dealt with in Lemma \ref{lemma amb vert not in cycle then not atomic}.
Below we consider the case where $\nu$ is in the initial cycle of $\Gamma_B$,  
and the case where it is in the terminal cycle is almost identical. 

Let $\mu$ be the vertex in the initial cycle of $\Gamma_{B}$ that joins the connecting path, let $\gamma$ be the vertex neighbouring $\mu$ on the connecting path.  Let $\pi$ be the path that starts at $\mu$, traverses the initial cycle twice, and ends by traversing the edge from $\mu$ to $\gamma$.    

Suppose $\sigma$ is the equivalence relation  in $\Sigma(\pi)$ that adds a new class each time an ambiguous vertex is entered except for the last time it enters $\nu$, when it adds to an inactive class.  Note that there will definitely be an inactive class the last time $\nu$ is entered as at least one new class has been added since the path last visited this vertex.  Let $\rho$ be the equivalence relation in $\Sigma(\pi)$ that adds a new class every time an ambiguous vertex is entered, including the last visit to $\nu$.

Aiming for a contradiction, suppose that there is an equivalence relation $\theta \in C$ such that $\sigma, \rho \leq \theta$.  Consider the path $\Pi(\theta)$, which exists since $|\theta| \geq |\sigma| \geq b$.  Since $\Pi(\sigma)=\Pi(\rho)=\pi$, both $\Pi(\sigma)$ and $\Pi(\rho)$ end at $\gamma$.  As $\Pi(\sigma), \Pi(\rho) \leq \Pi(\theta)$ and $\gamma$ is not in the initial cycle, the end vertices of $\Pi(\sigma)$ and $\Pi(\rho)$ must coincide in $\Pi(\theta)$. This means that when $\Pi(\theta)$ enters $\nu$ for the last time the new entry of $\theta$ must be added to both an inactive class and to a new class, a contradiction.  
Therefore $C$ does not satisfy the JEP and hence is not atomic by Proposition \ref{prop jep}. 
\end{proof}

We can now prove the first main result of this section, which is a characterisation of atomicity of $C=\Av(B)$ in terms of $\Gamma_B$:

\begin{thrm}
\label{thrm atomicity cons}
A finitely based avoidance set $C=\Av(B)$ of equivalence relations under the consecutive embedding ordering is atomic if and only if the following hold:
\begin{thmenumerate}
\item \label{thrm atomicity cons 0}
For each $\sigma \in C_{[1, b-1]}$ there is $\rho \in C_{b}$ such that $\sigma \leq \rho$; and
\item \label{thrm atomicity cons 1}
The factor graph $\Gamma_B$ is strongly connected or is a bicycle with no ambiguous vertices.
\end{thmenumerate}
\end{thrm}

\begin{proof}
($\Rightarrow$) Suppose $C$ is atomic.  To show that \ref{thrm atomicity cons 0} must hold consider an arbitrary 
$\sigma \in C_{[1, b-1]}$.  Take any $\theta \in C_{b}$.  Since $C$ is atomic, there must be an element $\gamma \in C_{[b, \infty)}$ containing both $\sigma$ and $\theta$.  Since $|\gamma|\geq b$, we can take any $b$ consecutive points, that include those of the embedding of $\sigma$,
to obtain  a relation of length $b$ containing $\gamma$.  Therefore \ref{thrm atomicity cons 0} holds.

Now we show that  \ref{thrm atomicity cons 1}  holds.  By Lemma \ref{lemma collection graph not atomic so is C}, $\Gamma_{B}$ is atomic so by Proposition~\ref{prop digraphs atomic}, either $\Gamma_{B}$ is strongly connected or $\Gamma_{B}$ is a bicycle.  
Furthermore, in the latter case, either $\Gamma_B$ is actually a cycle, in which case it is again strongly connected, or else it has no ambiguous vertices by Lemma \ref{lemma rest not atomic}.
This completes the proof of the forward direction.

($\Leftarrow$) Suppose \ref{thrm atomicity cons 0} and  \ref{thrm atomicity cons 1} hold.  
 Lemmas \ref{lemma strongly connected atomic} and \ref{lemma atomicity no amb verts} together imply that $C_{[b, \infty)}$ is atomic.  To extend this to $C$, let $\sigma, \rho \in C$ be arbitrary.  Then there exist $\sigma^{\prime}, \rho^{\prime} \in C_{[b, \infty)}$ such that $\sigma \leq \sigma^{\prime}$ and $\rho \leq \rho^{\prime}$, where \ref{thrm atomicity cons 0} is used if either $\sigma$ or $\rho$ has length $<b$.  By atomicity of $C_{[b, \infty)}$, there is an equivalence relation $\theta \in C_{[b, \infty)}$ such that $\sigma^{\prime}, \rho^{\prime} \leq \theta$.  Then $\sigma, \rho \leq \theta$ as well, meaning that $C$ satisfies the JEP and so is atomic by Proposition \ref{prop jep}.
\end{proof}

In order to turn the above characterisation into a decidability result, we need the following:

\begin{prop}
\label{prop decidability amb vert}
It is decidable whether a special vertex is ambiguous.
\end{prop}

\begin{proof}
Let $\nu$ be a special vertex in a factor graph $\Gamma_{B}$.  Suppose $\nu$ has $t$ equivalence classes and $\Gamma_{B}$ has $n$ vertices.

Suppose that $\nu$ is ambiguous.  Then there exists an edge $(\mu, \nu)$ and an equivalence relation $\sigma$ with an inactive class such that $\pi=\Pi(\sigma)$ ends at $\mu$.  Suppose that $\pi$ starts at $\rho$, an equivalence relation with $l$ classes.  Since $\nu$ is ambiguous, $\sigma$ must have at least $t$ classes, as if we were to extend $\Pi(\sigma)$ to $\nu$, there would be an inactive class to which an entry of the associated equivalence relation could be added.  This means that a new class is added to $\sigma$ in at least $t-l$ vertices of $\pi$; note that all such vertices must be special.  Let $\tau_{1}, \dots, \tau_{t-l}$ be the first $t-l$ of these vertices.  

Now we will describe a path $\pi^{\prime}$ of bounded length ending at $\mu$ such that there is an equivalence relation in $\Sigma(\pi^{\prime})$ with an inactive class.  We let $\pi^{\prime}$ be the path that starts at $\rho$, visits $\tau_{1}, \dots, \tau_{t-l}$ in order, and ends $\mu$, always taking the shortest route between these `stations'.  Each time we take the shortest route this path is of length $\leq n-1$, so the length of $\pi^{\prime}$ is $\leq (t-l+1)(n-1) \leq t(n-1)$, which is a constant not dependent on $\sigma$.  Any equivalence relation $\sigma^{\prime}$ constructed by traversing $\pi^{\prime}$ and adding a new class at $\tau_{1}, \dots, \tau_{t-l}$ has at least $t$ classes so must have an inactive class. 
So $\pi^{\prime}$ is a path of bounded length from which we can see that $\nu$ is ambiguous.

We have shown that if $\nu$ is ambiguous there is a path of length $\leq t(n-1)$ ending at $\mu$ that has an associated equivalence relation with an inactive class.  Therefore, to determine whether $\nu$ is ambiguous, we can examine all paths of length $\leq t(n-1)$ ending at $\mu$ and see if any have an associated equivalence relation with an inactive class.  If so, $\nu$ is ambiguous and, if not, $\nu$ is not ambiguous.  Since we have this bound on the length of paths to check, the decidability result follows.
\end{proof}

\begin{thrm}
\label{thrm decidability atomicity cons}
It is decidable whether a finitely based avoidance set $\Av(B)$ is atomic under the consecutive embedding ordering.
\end{thrm}

\begin{proof} 
The condition \ref{thrm atomicity cons 0} from Theorem \ref{thrm atomicity cons} is decidable since there are finitely many elements in $C_{[1, b-1]}$ and $C_{b}$ and so we can check whether all elements of $C_{[1, b-1]}$ are contained in elements of $C_{b}$.  It is also decidable whether $\Gamma_{B}$ is strongly connected or a bicycle, and by Proposition \ref{prop decidability amb vert} it is decidable whether $\Gamma_{B}$ contains ambiguous vertices.  Therefore the conditions of Theorem \ref{thrm atomicity cons} are decidable, and the result follows.
\end{proof}

\begin{ex}
\label{ex strongly conn atomic}
Consider $C=\Av(\lvertb1 \h 2 \h 3\lverte, \lvertb1 \h 3\lvertm2\lverte)$, as in Example \ref{ex section 4 summary 2}, whose factor graph is shown in Figure \ref{ex 4.27 figure}.  The equivalence relations in $C_{[1, b-1]}=C_{[1, 2]}$ are $\lvertb1\lverte, \lvertb1 \h 2\lverte$ and $\lvertb1\lvertm2\lverte$.  It can be seen that $\lvertb1\lverte, \lvertb1\lvertm2\lverte, \lvertb1 \h 2\lverte \leq \lvertb1 \h 2\lvertm3\lverte \in C_{3}$ so \ref{thrm atomicity cons 0} holds from Theorem \ref{thrm atomicity cons}.  Since the factor graph of $C$ is strongly connected, Theorem \ref{thrm atomicity cons} gives that $C$ is atomic.
\end{ex}

\begin{ex}
\label{ex bic no amb verts atomic}
The avoidance set $C=\Av(\lvertb1\lvertm2\lvertm3\lverte, \lvertb1 \h 2\lvertm3\lverte)$ from Example \ref{ex section 4 summary 3} has a factor graph which is a bicycle, shown in Figure \ref{ex 4.28 figure}.  The equivalence relations  in $C_{[1, b-1]}=C_{[1, 2]}$ are $\lvertb1\lverte, \lvertb1 \h 2\lverte$ and $\lvertb1\lvertm2\lverte$.  Then condition \ref{thrm atomicity cons 0} holds since 
$\lvertb1\lverte, \lvertb1\lvertm2\lverte, \lvertb1 \h 2\lverte \leq \lvertb1\lvertm2 \h 3\lverte \in C_{3}.$
Moreover, the factor graph contains no special vertices, and therefore no ambiguous vertices, so $C$ is atomic by Theorem \ref{thrm atomicity cons}.
\end{ex}

\begin{ex}
\label{ex not atomic}
Let $B=\overline{\Eq}_{4} \backslash X$ where
\begin{equation*}
X=\{\lvertb1 \h 2 \h 3 \h 4\lverte ,\ \lvertb1 \h 2 \h 3\lvertm4\lverte ,\ \lvertb1\lvertm2 \h 4\lvertm 3 \lverte,\ \lvertb 1 \h 3 \h 4\lvertm2 \lverte,\ \lvertb1\lvertm2 \h 3 \lvertm4\lverte ,\ \lvertb1 \h 2 \h 4\lvertm3\lverte\}
\end{equation*}
and consider $C=\Av(B)$.  The factor graph of $C$ is shown in Figure \ref{ex 4.7 figure}; this graph is neither strongly connected nor a bicycle, so $C$ is not atomic by Theorem \ref{thrm atomicity cons}.
\end{ex}

\section{Concluding remarks and open problems}
\label{sec:conc}

A comparison between the main results of our paper, and it predecessor \cite{mr} is perhaps somewhat intriguing, and points to possible further investigations. In each paper both the atomicity and wqo problems are shown to be decidable for consecutive embedding orderings by translating them into the appropriate factor graphs. The definition of these factor graphs can be viewed as completely analogous between the two papers. However, the criteria for atomicity (\cite[Theorems 5.1, 6.7]{mr} and our Theorem \ref{thrm atomicity cons}) and wqo (\cite[Theorems 5.2, 7.20]{mr} and our Theorem \ref{thrm wqo cdn cons}) are all saying slightly different things. Underlying these differences is perhaps an even more intriguing difference in the notion of ambiguity: while in \cite{mr} this refers to paths, for us it is a property of vertices.

This, in the authors' opinion, justifies further investigation of these properties for consecutive embedding orderings:

\begin{que}
\label{qu:1}
Are the atomicity and wqo problems decidable for consecutive embedding orderings of: (a) digraphs; (b) tournaments; (c) partial orders.
\end{que}

\begin{que}
Does there exist a general framework encompassing the results of \cite{mr} and the present paper, as well as the structures listed in Question \ref{qu:1}?
\end{que}

\begin{que}
Does there exist a (preferably natural) collection of relational structures for which either the atomicity or wqo problems under the consecutive embeddings are not decidable? Can these structures be chosen to have a single relation in their signature? Or even a single binary relation?
\end{que}

\section*{Acknowledgements} 
The authors would like to thank the anonymous referee for their helpful suggestions to improve the paper.

\section*{Declarations}

\subsection*{Funding}
None.

\subsection*{Competing interests}
None.

\subsection*{Data availability}
Data sharing not applicable to this article as no datasets were generated or analysed during the current study.

\subsection*{Author contributions}
The authors have contributed in equal measure to all aspects of the article. 
Both authors read and approved the final manuscript.

\bibliographystyle{plain}

\begin{thebibliography}{1}

\bibitem{alm}
 A. Atminas, V. Lozin, and M. Moshkov,
  WQO is decidable for factorial languages, Inform. and Comput. 256 (2017), 321--333. 
  
  \bibitem{braun19}
S. Braunfeld, The undecidability of joint embedding and joint homomorphism for hereditary graph classes, Discrete Math. Theor. Comput. Sci. 21 (2019), Paper No. 9, 17 pp. 

  
  \bibitem{braun21}
   S. Braunfeld, The undecidability of joint embedding for 3-dimensional permutation classes, Discrete Math. Theor. Comput. Sci. 22 (2021), Paper No. 10, 20 pp. 
   
   
  
  \bibitem{cherlin11}
  G. Cherlin,
Forbidden substructures and combinatorial dichotomies: WQO and universality,
Discrete Math. 311 (2011), 1543--1584. 
  
  \bibitem{ding92}
  G. Ding, Subgraphs and well quasi-ordering, J. Graph Theory 16
(1992), 489--502.


\bibitem{eli16}
S. Elizalde, A survey of consecutive patterns in permutations, Recent trends in combinatorics, 601--618, IMA Vol. Math. Appl. 159, Springer, 2016.

\bibitem{eli18}
S. Elizalde, P.R.W. McNamara,
The structure of the consecutive pattern poset,
Int. Math. Res. Not. IMRN 2018, 2099--2134. 


\bibitem{fraisse00}
R. Fra\"{\i}ss\'{e}, Theory of relations,
Stud. Logic Found. Math. 145,
North-Holland Publishing Co., Amsterdam, 2000.


\bibitem{higman}
G. Higman, Ordering by divisibility in abstract algebras,
Proc. London Math. Soc. 2 (1952), 326--336. 


\bibitem{hodges-mt}
W. Hodges,
Model theory,
Encyclopedia of Mathematics and its Applications 42, CUP, Cambridge, 1993. 

\bibitem{hucz15}
S. Huczynska, N. Ru\v{s}kuc,
Well quasi-order in combinatorics: embeddings and homomorphisms, Surveys in combinatorics 2015, 261--293,
LMS Lecture Note Ser. 424, CUP, Cambridge, 2015. 

\bibitem{liu20}
C-H. Liu,
Recent progress on well-quasi-ordering graph, Well-quasi orders in computation, logic, language and reasoning---a unifying concept of proof theory, automata theory, formal languages and descriptive set theory, 161--188,
Trends Log. Stud. Log. Libr., 53, Springer, Cham, 2020.

\bibitem{mr}
M. McDevitt and N. Ru\v{s}kuc, 
Atomicity and well quasi-order for consecutive orderings on words and permutations, SIAM J. Discrete Math. 35(1) (2021), 495-520.

\bibitem{rob-sey-04}
N. Robertson and P.D. Seymour,
Graph minors. XX. Wagner's conjecture,
J. Combin. Theory Ser. B 92 (2004), 325--357. 

\bibitem{vatter}
 V. Vatter, Permutation classes, Handbook of enumerative combinatorics, 753--833, Discrete Math. Appl. (Boca Raton), CRC Press, Boca Raton, FL, 2015.


\end{thebibliography}

\end{document}